\documentclass[a4paper,american,11pt, reqno]{amsart}

\usepackage[latin1]{inputenc}

\usepackage{amsmath, bbm}
\usepackage{amssymb}
\usepackage{hyperref}
\usepackage{url}
\usepackage{babel}
\usepackage{tikz}
\usepackage{pgf,tikz,pgfplots}
\usepackage{graphicx}

\usetikzlibrary{arrows}

\usetikzlibrary{calc}
\usepackage{amsfonts}
\input xy
\xyoption{all}

\newcommand{\Z}{{\mathbb Z}}

\newcommand{\C}{{\mathbb C}}

\newcommand{\R}{{\mathbb R}}

\newcommand{\A}{{\mathcal A}}

\newcommand{\csm}{{\text{csm}}}

\newcommand{\F}{{\mathcal F}}

\DeclareMathOperator{\Star}{\text{star}}
\DeclareMathOperator{\rank}{\text{rank}}

\DeclareMathOperator{\relint}{\text{int}}
\DeclareMathOperator{\tr}{\text{tr}}

\DeclareMathOperator{\crem}{\text{Crem}}
\DeclareMathOperator{\cl}{\text{cl}}

\newtheorem{thm}{Theorem}[section]
\newtheorem*{thmintro}{Theorem}
\newtheorem*{corintro}{Corollary}

\newtheorem{definition}[thm]{Definition}
\newtheorem{prop}[thm]{Proposition}
\newtheorem{proposition}[thm]{Proposition}
\newtheorem{lemma}[thm]{Lemma}
\newtheorem{cor}[thm]{Corollary}
\newtheorem{remark}[thm]{Remark}
\newtheorem{corollary}[thm]{Corollary}

          {\theoremstyle{definition}
}
          {\theoremstyle{definition}

\newtheorem{example}[thm]{Example}
}

 \numberwithin{equation}{section}

\tikzset{%
  add/.style args={#1 and #2}{to path={%
 ($(\tikztostart)!-#1!(\tikztotarget)$)--($(\tikztotarget)!-#2!(\tikztostart)$)%
  \tikztonodes}}
}

\newcommand{\comment}[1]{}

\begin{document}
\title{On the birational geometry of matroids}

\maketitle
\centerline{Kris Shaw and Annette Werner}
\begin{abstract}

This paper investigates isomorphisms of Bergman fans of matroids respecting different fan structures, which we regard as matroid analogs of birational maps. We show that  isomorphisms respecting the fine fan structure are induced by matroid isomorphisms. 

We introduce Cremona automorphisms of the coarse structure of Bergman fans, which are not induced by matroid automorphisms. 
We show that the automorphism group of the coarse fan structure is generated by matroid automorphisms and Cremona maps in the case of rank $3$ matroids which are not parallel connections and for modularly complemented matroids.

\end{abstract}
\centerline{MSC(2020):  14T20, 52B40, 14E07}
\section{Introduction}

This paper investigates the following question: 
What are the isomorphisms of Bergman fans of matroids with some choice of fan structures?
Some, but not all, isomorphisms of Bergman fans of matroids arise from isomorphisms of matroids.
For matroids which are realizable as an essential hyperplane arrangement in projective space, birational morphisms which are regular on the arrangement complement give rise to invertible linear maps between supports of the corresponding Bergman fans. Therefore we regard isomorphisms of Bergman fans as examples of matroid analogs of  birational maps. 

Recently, there have been exciting breakthroughs in matroid theory motivated by algebraic geometry. In \cite{AHK}, \cite{ArdilaDenhamHuh}, \cite{AminiPiquerez} a framework of Hodge theory for matroids was developed with applications to deep open combinatorial problems. For an overview see \cite{Ardila}. One key ingredient in these results is the study of Bergman fans and  the Chow rings of matroids. These tools also play a decisive role in the present paper.

 To be precise, let $M$ and $M'$ be matroids on ground sets $E$ and $E'$ respectively. 
For  $k \in E$, let  $v_k  \in \Z^E$ be the canonical basis vector associated to  $k$.
The Bergman fan of $M$ is denoted $B(M)$ and is viewed as a subset of $\R^{E} / \R \mathbbm{1}$, where $\mathbbm{1} = \sum_{k \in E} v_k$.
 We use the notation $B_\ast(M)$ to denote a chosen fan structure, 
see Section \ref{sec:MatFans}. We mostly investigate the coarse and fine fan structures, denoted 
 $B_c(M)$ and $B_f(M)$ respectively, as defined in Section \ref{sec:MatFans}.

 An isomorphism of Bergman fans  $ B_\ast(M)$ and  $B_\ast(M')$, where $\ast$ denotes  choices of fan structures, is defined as a linear map $\phi: \R^E / \R \mathbbm{1}  \to \R^{E'}/ \R \mathbbm{1} $ derived from a lattice isomorphism  $\Z^E \to \Z^{E'}$, such that $\phi$ restricts to an isomorphism of fans  $ B_\ast(M) \to B_\ast(M')$, see Definition \ref{def:morphism}.

Note that for simple matroids $M_1$, $M_2$ of rank $2$, every isomorphism between their Bergman fans (where all fan structures coincide) is obviously induced by a matroid isomorphism.  
Our first main theorem establishes the analogous claim for simple matroids of all ranks, if we consider the fine structure:

\begin{thmintro}[{\bf \ref{thm:fineMatroidAuto}}] Let $M_1$ and $M_2$ be simple matroids such that  neither is totally disconnected. If  $\phi: B_f(M_1) \to B_f(M_2)$
is an isomorphism of fans, then $\phi$ is induced by a matroid isomorphism. 
\end{thmintro}

If $M$ is not a simple matroid on the ground set $E$, let $\hat{M} $ denote its simplification on the ground set $\hat{E}$. Note that the affine Bergman fan of $M$ lives in a vector subspace of $\R^{E}$ which is determined by the parallel elements of $M$ and which is isomorphic to $\R^{\hat{E}}$. 
Moreover, the coordinate projection $\R^{E} \to \R^{\hat{E}}$ restricts to an isomorphism on this subspace and gives an isomorphism of the Bergman fan of $M$ with the Bergman fan of its simplification.  
Therefore, in order for the above theorem to hold,
we need to assume that $M_1$ and $M_2$ are simple. 
Moreover, the connectedness assumption in Theorem \ref{thm:fineMatroidAuto} is essential as the following example shows.

\begin{example}\label{ex:cremUniform}

If $M$ is totally disconnected and simple, then $M = U_{n+1, n+1}$, so that the lattice of flats is the Boolean lattice on $n+1$ elements. Let $\phi : \R^{E} \to \R^{E}$ be the linear map sending $v \mapsto - v$. This map  descends to a fan automorphism $\phi : B_f(M) \to B_f(M) $ which does not arise from a matroid automorphism. 
\end{example} 

The proof of the Theorem \ref{thm:fineMatroidAuto} is based on the observation that $\phi$ induces an isomorphism of Chow rings. In Section \ref{sec:Chowrings}, 
we review the Chow ring of fans and the degree map in the case of matroid fans, see \cite{AHK} and more generally \cite{AminiPiquerez}. We show that an isomorphism of Bergman fans of matroids with arbitrary fan structures induces an isomorphism of the respective Chow rings which is compatible with the degree map, see Proposition \ref{cor:chowdeg}. 
This isomorphism and Eur's degree formula \cite{eur} are the main tools used to prove Theorem \ref{thm:fineMatroidAuto}.

In Section \ref{sec:Orlik}, we show that the existence of isomorphisms of  Bergman fans implies that the underlying matroids have the same characteristic  polynomials. It is an open question, which non-isomorphic matroids have the same characteristic polynomials. As an example, the operation of parallel connection discussed in Section  \ref{sec:MatFans}   produces examples of non-isomorphic matroids with the same characteristic polynomials \cite{EschenbrennerFalk}. In Section \ref{sec:Orlik}, we also  conclude that the Chern-Schwartz-MacPherson classes of matroids from \cite{LdMRS} 
are preserved under matroid fan isomorphisms, see Proposition \ref{prop:samecsm}.

We then consider our main question for coarse fan structures. 
 In rank $3$ we show that there are no isomorphisms $\phi: B_c(M_1) \to B_c(M_2)$ for non-isomorphic matroids $M_1$ and $M_2$ , see Theorem \ref{prop:rank3}. However, the situation is different from rank $4$ on. 
 In Section \ref{sec:isosCoarse}, we show that there are  non-isomorphic matroids arising from parallel connections with isomorphic coarse Bergman fans, see Corollary \ref{cor:isoparallel}.
The hyperplane arrangement analogue of the operation of parallel connection produces non-isomorphic arrangements with diffeomorphic complements \cite{EschenbrennerFalk}.

We then switch perspective to automorphisms of a single Bergman fan and ask: What can be said of the automorphism group of a Bergman fan of a matroid with a fixed fan structure? Inspired by birational geometry, we introduce and study Cremona maps for matroids. These are automorphisms of the coarse Bergman fan not induced by matroid automorphisms.
The map $\phi$ from Example \ref{ex:cremUniform} is the first  simple example of a Cremona map, and is the tropicalization of the standard Cremona map on projective space. For connected matroids, we provide an intrinsic definition of Cremona maps. As an example, we describe in
Section \ref{sec:realizable} the automorphism group of the coarse Bergman fan of the braid arrangement matroid. Note that the Bergman fan of the braid arrangement in dimension $n$ is the moduli space of tropical genus zero curves with $(n+2)$-marked points. It was shown in \cite{ap}, that the automorphism group of this tropical moduli space is isomorphic to the symmetric group $S_{n+2}$. We show in 
Proposition \ref{lem:braid}, that this implies that the automorphism group of the coarse Bergman fan of the braid arrangement matroid is generated by the matroid automorphisms and a single Cremona map.  Theorem \ref{thm:fineMatroidAuto} implies that in this example the automorphism group of the coarse fan structure is strictly larger than the automorphism group of the fine fan structure.  

In general, we define Cremona maps as follows. If
$M$ is a simple matroid on the ground set $E$ with basis  $b = \{b_0, \dots, b_{d}\}$, 
we
define the flat $B_j$ as the closure of $b \backslash \{ b_j\}$ and put 
$v_{B_j} = \sum_{k \in B_j} v_k \in \R^E$.
Then we define the $\Z$-linear map $\crem_{b} \colon \R^{E}   \to \R^{E} $  by  
$v_{b_j}   \mapsto v_{B_j} $ for the basis elements $b_j$, and 
$v_k  \mapsto v_k $ for all $k \in E \backslash b$. 
If $\crem_b$ maps the line $\R \mathbbm{1}$ to itself, we also write 
$\crem_b: \R^{E } / \R \mathbbm{1} {\rightarrow} \R^{E} / \R \mathbbm{1}$ for the quotient map.  
 We then prove the following necessary and sufficient criterion for the map $\crem_b$ to be an automorphism of the coarse fan structure of the Bergman fan.

\begin{thmintro}[\bf \ref{thm:cremiff}] Let $b$ be a basis of a simple connected matroid $M$.  For any pair of elements $i, j \in E$, let $F_{i j}$ be the rank $2$ flat which is the closure of $i$ and $j$. The map $\crem_b$ descends to a linear map $\crem_b:  \R^{E} / \R \mathbbm{1} \rightarrow \R^{E} / \R \mathbbm{1}$ mapping  $B(M)$  to itself, if only if the sets 
$\{F_{i j} \backslash \{i, j\}\}_{i, j \in b}$ partition the set $E \backslash b$ into pairwise disjoint subsets.  
\end{thmintro}

We then focus on the case of rank three matroids in Section \ref{sec:rank3}. In particular, Theorem \ref{prop:rank3} states that if there is an isomorphism of coarse Bergman fans of rank three matroids,  then the underlying matroids are isomorphic. 
From this we can also deduce the following corollary. 

\begin{corintro}[\bf \ref{cor:autgroup_rank3}] The automorphism group of $B_c(M)$ for a rank $3$ matroid $M$ which is not a parallel connection is generated by matroid automorphisms and Cremona maps. 
\end{corintro} 

We conclude the paper by studying modularly complemented matroids. A simple modularly complemented matroid of rank at least $4$ is either a sufficiently big submatroid of projective space or a Dowling matroid associated to a group $G$. For the modularly complemented submatroids of projective space we prove in  Section \ref{sec:modular} that the automorphism group of $B_c(M)$ is equal to the group of matroid automorphisms. For Dowling matroids however, we find a Cremona transformation in the automorphism group of the coarse Bergman fan. This leads to the following theorem, combining Proposition \ref{prop:projective} and Proposition \ref{prop:dowling}.

\begin{thm}
For every simple modularly complemented matroid $M$, the automorphism group of $B_c(M)$ is generated by Cremona maps and matroid automorphisms. 
\end{thm}

Based on our results, one might ask: under which conditions
is the automorphism group of the coarse Bergman fan of  a simple matroid $M$ generated by matroid automorphisms and Cremona maps?

{\bf Acknowledgements: } The research of the first author was supported by the Trond Mohn Foundation project ``Algebraic and topological cycles in tropical and complex geometry". The second author was partially supported by the Deutsche Forschungsgemeinschaft (DFG, German Research Foundation) TRR 326 \textit{Geometry and Arithmetic of Uniformized Structures}, project number 444845124.

\section{Matroids and Bergman fans}\label{sec:MatFans}
A matroid is a pair $M = (E, r)$ where $E$ is a finite set  and $r : \mathcal{P}(E) \to \Z_{\geq 0}$ is a rank function satisfying the axioms

\begin{enumerate}
\item $r(A) \geq 0 $ for all $A \subseteq E$ and $r(\emptyset) = 0$,
\item For $A, B \subseteq E$ we have $r(A \cup B) + r(A \cap B) \leq r(A) + r(B)$, 
\item For any $i \in E$ and $A \subseteq E$, we have $r(A) \leq r(A \cup \{ i \}) \leq r(A) + 1$. 
\end{enumerate}

The rank of a matroid $M$ is $r(E)$ and is denoted $r(M)$. 
A \emph{flat} of a matroid $M$ is a subset $F \subseteq E$ which is closed under the rank function, which means that for all $i \notin F$ the rank of $F \cup \{ i\}$ is strictly greater than the rank of $F$. 
The flats are ordered by inclusion and form a lattice, in the sense of partially ordered sets. 
For every subset $S$ of $E$ we write $\cl(S)$  for the inclusion minimal flat containing $S$.

A \emph{circuit} of a matroid is a minimally dependent subset in the sense of inclusion. 
A \emph{loop} of a matroid $M = (E, r)$   is a element $i \in E$  such $r(\{i\}) = 0$. A pair of elements $i, j \in E$ are \emph{parallel} if neither $i$ nor $j$ is a loop and $r(\{i,j\}) = 1$. 
A matroid is \emph{simple} if it contains no loops or parallel elements. 
A \emph{connected component} of a matroid is an equivalence class under the relation defined by  $i \sim j$ if $ i$ and $j$ are contained in a circuit. Each connected component forms a matroid and a matroid decomposes into a direct sum of its connected components.
A matroid $M$ is \emph{totally disconnected} if all of its connected components have rank less than or equal to one. 
The \emph{parallel connection} $P_p(M_1, M_2)$ of two matroids $M_1$ and $M_2$ such that $E(M_1) \cap E(M_2) = \{p\}$ is a matroid on the set $E(M_1) \cup E(M_2)$, see \cite{oxley}, Section 7.1. We say that $M$ is a non-trivial parallel connection along $p \in E(M)$ if there exist matroids $M_1$ and $M_2$, which are both of rank at least two, so that $M$ 
is isomorphic to
$P_p(M_1, M_2)$. 
We say that $M$ is a \emph{non-trivial parallel connection}, if there exists some $p \in E(M)$ such that $M$ is a non-trivial parallel connection along $p$.

\begin{definition}
Let $M_1 = (E_1, r_1)$ and $M_2 = (E_2, r_2)$ be two matroids. A matroid isomorphism $f \colon M_1 \to M_2$ is a bijection $f \colon E_1 \to E_2$ such that for every subset $I \subseteq E_1$ we have $r_2(f(I)) = r_1(I)$. 
\end{definition}

Since we are interested in Bergman fans,  we will generally assume that our matroids are loopfree.
Suppose $M = (E, r) $ is a matroid  containing no loops.  
For every subset $F \subseteq E$, define a vector $v_F:= \sum_{i \in F} v_i \in \R^{E}$. 
We also write  $v_E= \mathbbm{1}$.
 
The (affine) \emph{Bergman fan} $\tilde{B}(M)$ associated to a loopfree matroid $M$ is defined as the subset of $\R^{E}$ consisting of all vectors $\sum_i a_i v_i$ such that for each circuit $C$ of $M$ the minimum of the set $\{a_i: i \in C\}$ is attained twice, see e.g.  \cite{feistu}. 
It is invariant under scaling by the vector $\mathbbm{1}$. We will often work with the \emph{projective Bergman fan} $B(M)$, which is defined  as the image of the affine Bergman fan  under the quotient map
$\R^{E} / \R \mathbbm{1}$.

\begin{remark}\label{rem:changeofcoordinates}
Given a polyhedral fan $\Sigma$ in a vector space $\R^n$ we may ask if $\Sigma$ is equal to the Bergman fan of a matroid with some fan structure up to a transformation in $\text{GL}_n(\Z)$. 
Studying the collection of such transformations in $\text{GL}_n(\Z)$ provides an alternative point of view to our main question. If  $A, B \in \text{GL}_n(\Z)$ are such that $A(\Sigma)$ and $B(\Sigma)$ are both Bergman fans of matroids with some fan structures, then there is an isomorphism of Bergman fans 
$\phi: A(\Sigma) \to B(\Sigma).$ The underlying matroids of  these two fans need not be the same. 

Tropical manifolds are locally modelled on matroid fans  up to coordinate changes in $\text{GL}_n(\Z)$  \cite{MikhalkinRau}. Therefore, answering our main question also has applications to tropical geometry. 
\end{remark}

The set $B(M)$ carries several interesting fan structures which have been investigated by Feichtner and Sturmfels \cite{feistu}:

i) We denote by $B_c(M)$ the natural fan structure  as a subfan of the normal fan of the matroid polytope \cite{feistu}, Proposition 2.5.  We refer to it as the coarse structure, since it is the coarsest fan structure on the set $B(M)$, see \cite{macstu}, Section 4.2 or \cite{hampe}, Proposition 3.4.1.

ii) The minimal nested set structure, which we denote by $B_m(M)$. The rays of the minimal nested set fan structure on $B(M)$ correspond to flats of $M$ which are \emph{connected}. A flat is connected if $M|F$ is connected. The cones of $B_m(M)$  correspond to \emph{nested} collections  of connected flats. A collection $\mathcal{S}$  of connected flats is called nested if for any subcollection of incomparable flats $F_1, \dots, F_t \in \mathcal{S}$ the join is not connected. 

iii) The fine structure which we denote by $B_f(M)$. This subdivision was first described in \cite{ArdilaKlivans}.
Here the building set in the sense of \cite{feistu}  consists of all flats, i.e.~it is maximal. We can describe the cones in the fine structure as follows:  They are induced by flag of flats $$\mathcal{F} := \emptyset \subsetneq F_1 \subsetneq F_2 \subsetneq \dots \subsetneq \dots F_k \subsetneq E$$ as the
rational  cones $$\rho_{\mathcal{F}} := \langle v_{F_1}, v_{F_2} , \dots , v_{F_k} \rangle_{\mathbb{R}_{\geq 0}}.$$

iv) We write $B_\ast(M)$ for an arbitrary unimodular fan structure. In this case, the rays of $B_{\ast}(M)$ need not correspond to flats of $M$. These fan structures are studied in \cite{AminiPiquerez}. Another example on such a fan structure is the conormal variety of a matroid from \cite{ArdilaDenhamHuh}. 

Note that by \cite[Proposition 2]{feiyuz}, the Bergman fan of a matroid from any building set is a unimodular fan. This implies that $B_m(M)$ and $B_f(M)$ are both unimodular fans. 
The following result by Feichtner and Sturmfels will be useful to study isomorphisms of Bergman fans with the coarse structure.

\begin{lemma}[{\cite{feistu}}, Theorem 5.3] \label{lem:coarsemin} The minimal nested set fan structure $B_m(M)$ coincides with the coarse fan structure $B_c(M)$ on the Bergman fan, if the matroid $M|G/F$ is connected for all flats $F,G$ such that $G$ is connected and contains $F$. 
\end{lemma}

For a polyhedral fan $\Sigma$ in $\R^n$ and a point $w \in |\Sigma|$ we define the star of $\Sigma$ at $w$ to be the subset of $$\Star_w(\Sigma) = \{ x \in \R^n \ | \ \exists \epsilon >0  \ s.t. \ \forall  0 < \delta < \epsilon,  \    w + \delta x \in |\Sigma|\}.$$
If the fan is an affine Bergman fan of a matroid $\tilde{B}(M)$, and $v_{\mathcal{F}}$ is the indicator of a flag of flats $\mathcal{F} = \{\emptyset = F_0 \subsetneq F_1 \subsetneq \dots \subsetneq F_k \subsetneq F_{k+1} = E\}$ of $M$ then by  \cite[Lemma 2.2]{FrancoisRau} and \cite[Proposition 2.2]{ArdilaKlivans} the star is the support of another affine Bergman fan, namely
\begin{equation} \label{eqn:stardescription}
\Star_{v_{\mathcal{F}}}(\tilde{B}(M)) =   \tilde{B}( M|F_{1}/F_{0}) \times  \dots \times \tilde{B}( M|F_{k+1}/F_{k}).
\end{equation}
It follows from this that the star of the projective Bergman fan $\Star_{v_{\mathcal{F}}}(B(M)) $ is then also a matroid 
fan $B\left( \bigoplus_{i = 1}^{k+1} M|F_{i}/F_{i-1}\right).$

\begin{lemma} \label{lem:parallelConnectionMissing}
Let $M$ be a simple matroid, and let $i$ be an element of $ E = E(M)$. The ray  $\langle v_i \rangle_{\mathbb{R}_{\geq 0}}$ is not a cone of some fan structure  $B_\ast(M)$ if and only if $M$ is a non-trivial parallel connection along $i$. 
\end{lemma}

\begin{proof} 
It suffices to consider the coarse fan structure. Suppose that $\langle v_i \rangle_{\mathbb{R}_{\geq 0}}$ is not a cone in $B_c(M)$. Then the star fan of $B_c(M)$
in direction $v_i$ has lineality space of dimension at least two. 
By Formula (\ref{eqn:stardescription}), the star of $B(M)$ at $v_i$ is the matroid fan  $B(M/i \oplus M|i)$.
 By \cite[Lemma 2.3]{FrancoisRau}, the dimension of the lineality space is one less than the number of connected components of the matroid $M/i \oplus M|i$. The matroid $M|i$ is a single coloop and hence connected.  Therefore,  the ray  $\langle v_i \rangle_{\mathbb{R}_{\geq 0}}$ is not in $B_c(M)$ if and only if  $M/i$ is disconnected. By \cite[Proposition 7.1.15 iii) and 7.1.16 ii)]{oxley}, the matroid $M/i$ is disconnected if and only if  $M$ is a non-trivial parallel connection along $i$. This completes the proof. 
\end{proof} 

Note that for simple $M$, every $i \in E$ gives rise to a ray in $B_m(M)$ and $B_f(M)$. 
If $B_{\ast}(M) $ comes from a building set, then all rays of the fan come from flats of $M$. In this case,  we define the rank of a ray in $B_\ast(M)$ as the rank of the associated flat. We use the same terminology for rays in the coarse structure $B_c(M)$.

\begin{definition}\label{def:morphism} 
Let $N_1$ and $N_2$ be lattices and $\Sigma_1$ and $\Sigma_2$ be rational polyhedral fans in $N_1 \otimes \R$ and $N_2 \otimes \R$, respectively.
A morphism of rational polyhedral fans $\phi : \Sigma_1 \to \Sigma_2$ is the restriction of a linear map induced by a lattice morphism $N_1 \to N_2$, which sends cones of $\Sigma_1$ to cones of $\Sigma_2$.  A morphism of fans is an isomorphism if it is a  bijective morphism whose inverse map is also a fan morphism.

We write $\mathrm{Aut}(\Sigma)$ for the group of automorphisms of $\Sigma$ in this sense.
\end{definition}

Let $M_1, M_2$ be loopfree matroids on ground sets $E_1$ and $E_2$, respectively. Note that any isomorphism of matroids $f : M_1 \to M_2$ 
gives rise to  a $\Z$-linear map $\phi_f: \R^{E_1} \to \R^{E_2}$ satisfying $\phi_f(v_i) = v_{f(i)}$. The  isomorphism  descends to the quotients, and the map $\phi_f$ induces an isomorphism of the fine Bergman fans $\phi_f: B_f(M_1) \to B_f(M_2)$ and also of the coarse Bergman fans $\phi_f: B_c(M_1) \to B_c(M_2)$. 
In particular,  there is a map
$$\text{Iso}(M_1, M_2) \to \text{Iso}( B_\ast(M_1), B_\ast(M_2))$$
for $\ast$ denoting either the coarse or fine fan structure. 
If both matroids are simple, the map is injective.

More generally, if there is a matroid isomorphism $f: M_1 \to M_2$, then given any fan structure $B_\ast(M_1)$ on $B(M_1)$ there is a compatible fan structure $B_\ast(M_2)$ on $B(M_2)$ such that the induced map $\phi_f: B_\ast(M_1) \to B_\ast(M_2)$ is a fan  isomorphism. 

\section{Realizable matroids}\label{sec:realizable}
 If the matroid $M$  arises from an arrangement of hyperplanes $\mathcal{A}$ with trivial intersection in projective space over some field,  its Bergman fan is given by the tropicalization of the associated hyperplane complement $\Omega_{\mathcal{A}}$ (over a trivially valued ground field), which is a very affine variety sitting in the intrinsic torus.  
In this case,  every fan structure on the Bergman fan gives rise to a compactification of $\Omega_{\mathcal{A}}$ by taking the closure in the toric variety associated to the fan, see \cite{tevelev}, Proposition 2.3. In this way, the minimal nested set fan induces the (minimal) wonderful compactification introduced by de Concini and Procesi \cite{conpro}.
It is shown in \cite{kuwe} that for essential and connected hyperplane arrangements every dominant morphism $f: \Omega_{\mathcal{A}} \rightarrow \Omega_{\mathcal{A}}$ extends to an  endomorphism of the visible contour compactification, which is induced by the coarse structure on the Bergman fan. Hence it extends to an endomorphism of the (minimal) wonderful compactification whenever the criterion in Lemma \ref{lem:coarsemin} by Feichtner and Sturmfels is satisfied.
In fact, the proof of \cite{kuwe}, Theorem 5.1 shows that every dominant endomorphism $f$ of $\Omega_{\mathcal{A}}$  extends to a endomorphism of the intrinsic torus which is a homomorphism up to translation. Hence by passing to the tropicalization we get a linear map preserving the Bergman fan, i.e.~an endomorphism of $B_c(M)$  in the sense of Definition \ref{def:morphism}. 
Therefore we regard linear maps preserving the different fan structures on the Bergman fan as examples of analogs of birational maps between matroids. 

We will now look at an interesting example of a realizable matroid, namely the matroid given by the $A_n$ type root system:

\begin{example}\label{ex:braid} {\bf The braid arrangement.} The root hyperplanes of the root system of type $A_n$ for $n \geq 2$ induce (after passing to the quotient by their intersection) the  essential arrangement in $\mathbb{P}^n_{K}$ (where $K$ is any algebraically closed field) given by the hyperplanes
$V(x_i)$ for $i = 0 \ldots, n-1$ and $V(x_i - x_j)$ for $i < j$ in $\{0, \ldots, n-1\}$. 
This arrangement $\mathcal{A}$  is called the braid arrangement. 
 Mapping the class of a point $(p_1, \ldots p_{n-1}, 1)\in \mathbb{G}_{m,K}^n$  to the points $p_1, \ldots, p_{n-1}, 1, 0, \infty$ in $\mathbb{P}^1_K$ induces an isomorphism between the  associated hyperplane complement $\Omega_{\mathcal{A}}$ in $\mathbb{G}_{m,K}^n$ and  the moduli space $M_{0,n+2}$ of $n+2$-pointed curves of genus zero. 
The minimal wonderful compactification of $\Omega_{\mathcal{A}}$ is isomorphic to the Deligne-Mumford compactification $\overline{M}_{0, n+2}$. It is shown in \cite{brunomela} that the automorphism group of $\overline{M}_{0, n+2}$ is isomorphic to $S_{n+2}$ and hence equal to the automorphism group of $\Omega_{\mathcal{A}}$. The new transpositions in the symmetric group are realized by Cremona automorphisms.

The matroid $M(\mathcal{A})$ induced by this arrangement is isomorphic to the graphic matroid given by the complete graph $K_{n+1}$. This implies that the automorphism group of $M(\mathcal{A})$ is equal to the symmetric group $S_{n+1}$, i.e.~the Weyl group of the root system.  The Bergman fan of the  matroid $M(\mathcal{A})$ can be identified with the moduli space $M_{0,n+2}^{trop}$ of tropical genus zero curves with $n+2$ marked points, see \cite{ArdilaKlivans}, chapter 4 and \cite{FrancoisRau}, Example 7.2.  In \cite{ap} it is shown that the automorphism group of $M_{0,n+2}^{trop}$ (and also the automorphism group of its compactification) is equal to $S_{n+2}$, i.e.~it is the same as the automorphism group in the algebraic case. This implies that  $Aut(B_c(M(\mathcal{A})))$ is  $S_{n+2}$.
\end{example} 

\begin{proposition}\label{lem:braid} Let $M(\mathcal{A})$ be the matroid given by the braid arrangement. Then the automorphism group of the coarse Bergman fan $B_c(M(\mathcal{A}))$ is generated by the matroid automorphisms and a combinatorial Cremona map. 
\end{proposition}

\begin{proof} 
We have to make some of the identifications discussed above explicit. Let $M = M(\mathcal{A})$ be the matroid given by the braid arrangement for $A_n$. First of all, let us identify $E= E(M)$ with the set $[\binom{n}{2}] = \{1, \ldots, \binom{n}{2} \}$ (where we write two-element subsets as ordered pairs) by the following map:
\[H_{ij} \mapsto (i+1, j+1) \mbox{ for all } i < j \mbox{ in } \{0, \ldots, n-1\}\]
and 
\[H_i \mapsto (i+1,n+1) \mbox{ for all } i \mbox{ in }  \{0, \ldots, n-1\}.\]
The elements $b_i = (i,n+1)$ for $i = 1, \ldots n$  form a basis of $M$, which satisfies the condition from Theorem \ref{thm:cremiff}. 
Hence $\crem_b$ is an automorphism of $B_c(M)$. 

The identification of $B_c(M)$ with $M_{0,n+2}^{trop}$ (e.g. spelled out in \cite{chmr}, Section 2.1.2) now shows that $\crem_b$ induces the transposition of $n+1$ and $n+2$ in the natural action of $S_{n+2}$ on $M_{0,n+2}^{trop}$. Moreover, the automorphism group $S_{n+1}$ of $M$ acts on $M_{0,n+2}^{trop}$ as permutation of the first $n+1$ markings. This implies our claim. \end{proof} 

\begin{example}\label{ex:4points}
Notice that in general not all automorphisms of $B(M(\A))$ lift to automorphisms of $\Omega_{\A} $.  Consider for example the arrangement $\A = \{0, 1, \infty, p\}$ on $\mathbb{P}^1$. The automorphism group of $B(M(\A))$ is the symmetric group $S_4$. 
On the other hand, any automorphism of $\Omega_{\A}$ extends to an automorphism  of $\mathbb{P}^1$ and is therefore  contained in  $\text{PGL}_2$ and thus preserves the cross ratio of $(0, 1, \infty, p)$.  On the other hand, for appropriate choices of permutations in $S_4$ and points $p$ the cross ratio will not be preserved. Hence automorphisms of Bergman fans do not always arise from automorphisms of complements.
\end{example}

\section{Chow rings of matroids}\label{sec:Chowrings}

Let $\Sigma$ be a rational polyhedral fan in $ N_{\R}= N \otimes \R$ where $N$ is a rank $n$ lattice. Let $N^*$ denote the dual lattice. 

\begin{definition}\label{def:ChowRing}
The combinatorial Chow ring of $\Sigma$ is the  ring

$$A^*(\Sigma) = \frac { \mathbb{Z} [x_{\rho} \ | \ \rho \in \Sigma^1 ]} { I + J},$$
where $I$ is the ideal generated by squarefree monomials coming from non-faces of $\Sigma$, i.e.~ 
$$I = ( x_{\rho_1} \dots x_{\rho_k} \ | \   \rho_1, \dots, \rho_k \in \Sigma^1 \text{ pairwise distinct, not spanning a cone }) $$ 
and $J$ is the ideal 
$$J = \big (\sum_{\rho \in \Sigma^1}  m (v_{\rho}) x_{\rho} \ | \ m \in N^* \big),$$
where $v_{\rho}$ is the primitive integer vector in direction $\rho$. 
\end{definition}

An analogous ring was introduced by Danilov in \cite{Danilov} as the Chow ring of a complete non-singular toric variety. 
The combinatorial Chow ring is the Chow ring of a non-singular toric variety of a unimodular fan $\Sigma $ even when the fan is not complete by \cite[Section 3.1]{Brion}.  When the fan $\Sigma $  is unimodular, quasi-projective,  and has support $B(M)$  the ring $A^*(\Sigma)$ has Poincar\'e duality, Hard Lefschetz, and satisfies the Hodge Riemann bilinear relations \cite{AminiPiquerez}.

\begin{prop}\label{prop-chow}  An isomorphism of rational polyhedral fans $\phi: \Sigma_1 \to  \Sigma_2$  induces an isomorphism of graded rings  $\phi_*: A^*(\Sigma_1) \to A^*(\Sigma_2)$. 
 \end{prop}
\begin{proof}
Denote the Chow rings of $\Sigma_i$ by 
$$A^*(\Sigma_i) = \frac { \mathbb{Z} [x_{\rho} \ | \ \rho \in \Sigma_i^1 ]} { I_i + J_i},$$
where the ideals $I_i$ and $J_i$ are described above. 

The map $\phi$ maps rays of $\Sigma_1$ bijectively to rays of $\Sigma_2$ which induces a  
bijection between the generators of the two Chow rings. Moreover, the ideal $I_1$ from the non-faces of $\Sigma_1$ 
is sent exactly to the corresponding ideal $I_2$ of $\Sigma_2$. 
To show that $J_1$ maps to $J_2$ 
consider a relation
$ \sum_{\rho \in \Sigma_1^1}  m (v_{\rho}) x_{\rho}$ for some $m \in N^*$. 
Then $$ \sum_{\rho \in \Sigma_1^1}  m(v_{\rho}) x_{\phi(\rho)} = \sum_{\rho' \in \Sigma^1_2} m (\phi^{-1}v_{\rho'})   x_{\rho'} =  \sum_{\rho' \in \Sigma^1_2} \phi^*m (v_{\rho'})   x_{\rho'},$$
therefore $\phi(J_1) \subseteq J_2$. Reversing the argument shows that $\phi(J_1) =  J_2$. Hence the map on the Chow rings induced by $\phi$ is an isomorphism. 
\end{proof}

Let us now consider the fine fan structure on a Bergman fan of a loopfree matroid $M$. For brevity we will call the Chow ring of $B_f(M)$ simply the Chow ring of the  matroid $M$ and denote it by $A^*(M) $.  
In terms of generators and relations this ring is
$$A^*(M) = \frac { \langle x_F \ | \ F \text{ proper flat of } M\rangle_{\mathbb{Z}}}{ I + J},$$
where $I$ is the ideal generated by the monomials $x_Fx_G$ for $F \not \subseteq G$ and $G \not \subseteq F$. 

Set $E = \{0, \dots, n\}$, then a basis for  $N$ is   $\{v_1, \dots, v_n\}$, where $v_k \in N$ is the image of the standard basis vector of $\Z^{E}$ under the quotient. 
We claim  the ideal $J$ is generated by the degree $1$ polynomials
$$\sum_{ i \in F} x_F - \sum_{0  \in G} x_G,$$
for all $i \in E \backslash 0$. 
Recall that the matroid fan $B_f(M)$ is defined with respect to the lattice $N = \Z^{E} / \Z \mathbbm{1}$. A collection of generators of the ideal $J$ is obtained by considering the relations in Definition \ref{def:ChowRing} coming from a basis of $N^*$. The relation above is obtained by taking $v_k^* \in N^*$ in Definition \ref{def:ChowRing}.

This ring is the Chow ring of the  \emph{maximal wonderful compactification} of the complement of a hyperplane arrangement in projective space when $M$ is representable over $\C$, see \cite{conpro}, \cite{feiyuz}. The maximal wonderful compactification is obtained from projective space by blowing up all intersections of the hyperplanes  in $\A$ starting from the intersections which are points and proceeding by dimension.

If $M$ is a simple matroid of rank $d+1$,  then by \cite[Proposition 5.10]{AHK} 
there  is an isomorphism $$\deg : A^d(M)  \to \mathbb{Z}$$ such that $\deg(  x_{F_1} \dots x_{F_d}) = 1$ whenever $F_1 \subsetneq \dots  \subsetneq F_d$ is a maximal flag of flats.  

Moreover, there is the following formula for the degree map on monomials from  \cite{eur} in terms of the coefficients of the characteristic polynomial of a matroid. 
We write the reduced characteristic polynomial of a loopfree matroid $M$  as
$$\tilde{\chi}_M(t) = \mu^0(M) t^{rank M - 1} - \mu^1(M) t^{rank N - 2} + \dots \pm \mu^{rank M - 1}(M),$$
denoting its $k$-th unsigned coefficient by  $ \mu^{k}(M)$.
The reduced characteristic polynomial is always monic, hence $\mu_0 = 1$.

\begin{thm}\cite[Theorem 3.2]{eur} \label{formula:Eur}
Let $F_1  \subsetneq F_2 \subsetneq \dots \subsetneq F_k$ denote a flag of proper  flats of $M$ where $r_i = r(F_i)$ then for all $d_i\geq 0$ with $ r(M)-1 = d = \sum_{i = 1}^k d_i $ 
\begin{equation}
\deg(x_{F_1}^{d_1} \dots x_{F_k}^{d_k}) = (-1)^{d-k}\prod_{i = 1}^k \binom{d_i - 1 }{\tilde{d}_i - r_i} \mu^{\tilde{d}_i - r_i}(M| F_{i+1} / F_i)
\end{equation}
where $\tilde{d_i} = \sum_{j = 0}^i d_j$ 
\end{thm}

For more general unimodular  fan structures $\Sigma$ on $B(M)$, Amini and Piquerez show that 
the Chow ring of $\Sigma$ still satisfies  $A^d(\Sigma) \cong \mathbb{Z}$ where $d+1$ is the rank of $M$ \cite[Theorem 1.2]{AminiPiquerez}. Moreover, the isomorphism is given by the degree map $$\deg : A^d(\Sigma)  \to \mathbb{Z}$$
sending $x_{\rho_1}x_{\rho_2} \dots x_{\rho_d}$ to $1$, whenever $\rho_1, \dots, \rho_d \in \Sigma^1$ are rays of a $d$-dimensional face of $\Sigma$, see \cite[Section 3.4]{AminiPiquerez}.

\begin{proposition}\label{cor:chowdeg} Let $M_1$ and $M_2$ be loopfree matroids, and 
let $\Sigma_1, \Sigma_2$ be  unimodular fans of dimension $d$ whose supports are respectively $B(M_1)$ and $B(M_2)$. If $\phi: \Sigma_1 \to \Sigma_2$ is an isomorphism of fans then there exists an isomorphism of Chow rings $\phi_{\ast} : A^*(\Sigma_1) \to A^*(\Sigma_2)$ which is compatible with the degree maps. More precisely, for any $\alpha \in A^d(\Sigma_1)  $ we have $\deg (\alpha ) = \deg (  \phi(\alpha)).$
\end{proposition}
\begin{proof} By Proposition \ref{prop-chow}, the isomorphism $\phi$ of fans induces an isomorphism $\phi_{\ast} : A^*(\Sigma_1) \to A^*(\Sigma_2)$ of Chow groups. The map $\phi_{\ast}$ is compatible with the degree map since rays  $\rho_1, \dots, \rho_d $ of $\Sigma_1$  generate a $d$-dimensional face of $\Sigma_1$ if and only if the rays  $\phi(\rho_1), \dots, \phi(\rho_d)$ generate a $d$-dimensional face of $\Sigma_2$.
\end{proof}

\section{Orlik-Solomon algebra and the characteristic polynomials of matroids}\label{sec:Orlik}
From the previous section and the intersection formula in Equation \ref{formula:Eur}, we 
see that the reduced characteristic polynomials of minors of $M$  determine the intersection numbers in $A_{\ast}(M)$. By Proposition \ref{prop-chow}, these intersection numbers must be preserved under an isomorphism of fine fan structures. 
In fact, we show in this section that if an isomorphism between any fan structures exists the matroids must have the same characteristic  polynomial.

The following proposition describes the  coefficients of the reduced characteristic polynomial in terms of the support of the Bergman fan of $M$. We define $$\mathcal{F}_{p}(B(M))  =\sum_{\sigma \subset B_{\ast}(M)} \bigwedge^{p} \langle \sigma \rangle  \subset \bigwedge^p (\R^E / \R \mathbbm{1}),$$
where $\langle \sigma \rangle  \subset \R^E / \R \mathbbm{1}$ denotes the linear span of the cone $\sigma$, and $B_{\ast}(M)$ denotes the choice of any fan structure on $B(M)$.
We set 
$\mathcal{F}_{0}(B(M)) = \R$.

\begin{proposition}\label{prop:OScoeff}
If $M$ is a loopfree matroid, then 
 $$\mu^p(M)  = \dim \mathcal{F}_{p} (B(M)).$$   
\end{proposition}

\begin{proof}
By \cite[Theorem 4]{Zharkov}, the dual of $\mathcal{F}_{p} (B(M))$ is  isomorphic to $OS^p(M) $ where $OS^*(M)$ denotes the Orlik-Solomon algebra of $M$ over $\R$. Moreover, by \cite{OrlikTerao}  the reduced characteristic polynomial of $M$ is 
$$\tilde{\chi}_M(t) =  \sum_{p = 0}^{r(M) -1} (-1)^{r(M) -p-1} \dim OS^p(M)t^{r(M) -p-1}.$$ 
This completes the proof.
\end{proof}

\begin{example}\label{ex:mu1}
We can use Proposition \ref{prop:OScoeff} to determine the coefficient $\mu^1(M)$ for any simple matroid $M$. 
Since $M$ has no  parallel elements,  then the vectors $v_i$ are contained in $\tilde{B}(M)$ for any $i \in E$. Upon taking the projective Bergman fan we have  
$\dim \F_1(M) = |E| -1.$ Therefore, we have $\mu^1(M) = |E|-1$. 
\end{example}

\begin{proposition}\label{prop:samecharpoly}
If $\phi : B(M_1)  \to B(M_2)$ is an isomorphism of supports of Bergman fans, then $\dim \mathcal{F}_{p}(B(M_1)) = \dim \mathcal{F}_{p}(B(M_2))$ for all $p$, and $M_1$ and $M_2$ have the same reduced characteristic polynomial. 
\end{proposition}

\begin{proof}
The map $\phi$ induces isomorphisms on the  exterior powers $\phi_{\ast} : \bigwedge^p \R^{E_1} \to \bigwedge^p \R^{E_2}$. 
Moreover, since  $\phi$ preserves the supports of the fans it sends the generators of $\mathcal{F}_{p}(B(M_1)) $ to the generators of $\mathcal{F}_{p}(B(M_2))$ for all $p$. This proves the claim about the dimensions. 
This claim regarding the characteristic polynomial follows from Proposition \ref{prop:OScoeff}.
\end{proof}

Non-isomorphic matroids with the same characteristic polynomials have been studied in \cite{EschenbrennerFalk}. There are known examples of combinatorially equivalent arrangements of lines in $\C P^2$ with non-homotopic complements \cite{Rybnikov}. Since they are combinatorially equivalent,  the arrangements define the same matroid, and hence have the same characteristic polynomial. The above proposition implies that classifying matroids up to isomorphisms of their fans is a priori a finer notion of equivalence than up to equality of their characteristic polynomials. 

The Chern-Schwartz-MacPherson (CSM) cycles of a matroid were introduced in \cite{LdMRS}. Here they are defined as 
 Minkowski weights supported on the matroid fan equipped with the fine
structure. Recall that a $k$-dimensional Minkowski weight on a fan $\Sigma$ is an assignment $w : \Sigma^{k} \to \Z$ satisfying a \emph{balancing condition} on all faces of codimension one \cite{FultonSturmfels}.  In \cite{LdMRS},  the CSM cycles of matroids are defined as Minkowski weights on the fine fan structure of Bergman fans. Moreover, the weights of the cones are  described in terms of invariants of the underlying matroid in the following way.
Let $\mathcal{F} = \{ \emptyset \subsetneq F_1 \subsetneq \dots \subsetneq F_k \subsetneq E\}$ be a flag of flats of a rank $d+1$ matroid $M$, then the $k$-th CSM cycle of $M$  is the weight function  $w_{\csm_k(M)} : \Sigma^k \to \Z$ given by 
$$w_{\csm_k(M)}(\sigma_{\mathcal{F}}) = (-1)^{d-k} \prod_{i = 1}^{k} \beta ( M|F_{i}/F_{i-1}),$$
where $\beta$ denotes the beta invariant of $M$, namely
$\beta(M) = (-1)^d \tilde{\chi}_M(1)$.
The above formula determines
the underlying matroid $M$ and a priori the weight of a face may be different in the presence of matroid fan isomorphisms.  Moreover, the above definition of the CSM cycle is only a Minkowski weight for the fine fan structure. 

The next lemma provides an alternative formula for the  weights of the CSM cycles, and shows they can be recovered from the support of the fan. Hence this recipe produces CSM cycles as Minkowski weights on $B(M)$ with any fan structure. 
This recipe also allows us to consider the CSM cycles of matroids as fan tropical cycles. Tropical cycles are equivalence classes of Minkowski weights \cite{AllermanRau}. Two Minkowski weights 
$w_1: \Sigma_1^k \to \Z$ and $w_2: \Sigma_2^k \to \Z$ are equivalent and hence represent the same tropical cycle if there exists a fan $\Sigma_3$ which refines both $\Sigma_1$ and $\Sigma_2$ and a Minkowski weight $w_3 : \Sigma_3 \to \Z$
such that $w_1(\sigma_1) = w_2(\sigma_2) = w_3 (\sigma_3)$ if both $\sigma_1$ and $\sigma_2$ lie in a common face $\sigma_3$ of $\Sigma_3$.

To alternatively define the CSM cycles, first set
$$\tilde{\chi}_{\sigma}(t) = \sum_{i = 0}^d (-1)^{d-k} \dim \mathcal{F}_p(\Sigma_M(\sigma))$$ 
where 
$ \mathcal{F}_p(\Sigma_M(\sigma)) = \sum_{\sigma' \supset \sigma} \bigwedge^p \langle \sigma' \rangle$. 
This differs from the definition of the vector space
$ \mathcal{F}_p(\Sigma_M)$ in that it only takes into account faces containing $\sigma$. Notice again that the vector spaces  $\mathcal{F}_p(\Sigma_M(\sigma))$ are not dependent on the fan structure chosen and can be defined using any point $x \in \relint(\sigma)$. 

Now let  $\Sigma_M$ be an arbitrary fan structure on $B(M)$ for some loopfree matroid $M$. Then set
\begin{equation} \label{eqn:weightCSM}
w_{\csm_k(\Sigma_M)}(\sigma) = \frac{\tilde{\chi}_{\sigma}(t)}{(1-t)^k }|_{t=1}.
\end{equation}

\begin{lemma}
Let $M$ be a loopfree matroid and $B_f(M)$ be the Bergman fan with fine fan structure. Then
$$w_{\csm_k(M)} (\sigma_{\F})= w_{\csm_k(B_f(M) )}(\sigma_{\F})$$
for all flags of flats $\F$ of $M$. 

Moreover, for another fan structure $\Sigma_M'$ on  $B(M)$ the Minkowski weights $w_{\csm_k(\Sigma_M)} $ and $w_{\csm_k(\Sigma_M')} $
define the same tropical cycles. 
\end{lemma}

\begin{proof}
The first statement follows from the fact that the beta invariant of a matroid is equal to $\beta(M) = (-1)^d \tilde{\chi}_M(1)$ and 
$\chi_{M_1}(t)\chi_{M_2}(t) = \chi_{M_1 \oplus M_2}(t)$, where $\chi_M(t)$ denotes the non-reduced characteristic polynomial, 
that is $(1-t) \tilde{\chi}_M(t) = \chi_M(t)$. 

For the second statement, if $\Sigma_M$ and $\Sigma_M'$ are two fan structures on $B(M)$ then there exists a common refinement $\Sigma''_M$. Since the recipe for the weight $w_{\csm_k(\Sigma)}(\sigma) $ depends only on the support of $\Sigma$ locally about $\sigma$ the CSM Minkowski weights of all three fans are equivalent. 
\end{proof}

\begin{proposition}\label{prop:samecsm}
If $\phi : B_\ast(M_1)  \to B_\ast(M_2)$ is an isomorphism of fans where $\ast$ denotes arbitrary fan structures, then $\phi$ sends $\csm_k(M_1)$ to $ \csm_k(M_2)$ for all $k$.  \end{proposition}

\begin{proof}
To shorten notation let $\Sigma_i = B_\ast(M_i)$.  
Since $\phi$ is an isomorphism of fans it sends cones of dimension $k$ of $\Sigma_1$ to cones of dimension $k$ to $\Sigma_2$. It suffices to prove that the weight of a face $\sigma \in \Sigma_1^k$  in $\csm_k(M_1)$ is the same as the weight of $\phi(\sigma) \in \Sigma^k_2$  in $\csm_k(M_2)$. 
By the same proof as for Proposition \ref{prop:samecharpoly}, for all $p$ we have $\dim \F_p(\Sigma_1(\sigma)) = \dim \F_p(\Sigma_2(\phi(\sigma)))$. The statement now follows from the formula for the weights of faces in CSM classes in Equation \ref{eqn:weightCSM}.
\end{proof}

\section{Isomorphisms preserving the fine structure}
The goal of this section is to prove Theorem \ref{thm:fineMatroidAuto} which states that every isomorphism between Bergman fans preserving the fine structure is induced by a matroid isomorphism if the matroids are not totally disconnected.

\begin{lemma}\label{lem:equivtotallydisconnected}
If  $M$ is a loopfree  matroid, then the  following are equivalent:
\begin{enumerate}
\item  $M$ is totally disconnected.
\item $M$ does not contain a circuit of size  greater than or equal to $3$.
\item $M / F$ is totally disconnected for every rank $1$ flat $F$.
\item $B(M)$ is a linear space of dimension $r(M) -1$.
\item the unsigned constant term of  the (reduced) characteristic  polynomial $\mu_d(M)$ is  $1$. 
\end{enumerate} 
\end{lemma}

\begin{proof}
We first show the equivalence of statements $1)$ and $2)$.
The circuits of a  matroid which is totally disconnected are all of size less than or equal to $2$. 
Conversely, if $M$ 
is not totally disconnected then there must be a connected component of $M$ 
which contains a circuit of size at least $3$. This circuit is also a circuit of $M$. 

Statement $1)$ implies $3)$ directly. Suppose statement $3)$ holds and let $C$ be a circuit of $M$. If $C$ is contained in a rank $1$ flat of $M$, then $|C| =2$. Therefore,  suppose $C = \{ i_1, \dots, i_k\}$ with $i_j \in F_j$ where $F_j$ are distinct rank $1$ flats of $M$ and $k\geq 3$.
Then 
$r_{M / F_k}(\{i_1, \dots, i_{k-1}\})  = k-1 $ since $M /F_k$ is totally disconnected and $F_k$ is of rank $1$. This implies  that 
$r(C) = r_{M / F_k}(\{i_1, \dots i_{k-1}\}) + r(F_k) = k$ contradicting the fact that $C$ is a circuit. 
Therefore, all circuits of $M$ have size less than or equal to $2$ so $3)$ implies $2)$.

The equivalence of $1)$ and $4)$ follows from the fact that the  dimension of the lineality space of a Bergman fan is  one less than the number of connected components of the corresponding matroid, Proposition \cite[Lemma 2.3]{FrancoisRau}.
The equivalence of $4)$ and $5)$ follows from Proposition \ref{prop:OScoeff}.  
 This completes the proof. 
\end{proof}

\begin{lemma} \label{lem:localtotallydisconnected}
Let $M_1$ and $M_2$ be simple matroids, and let $i$ be in $E(M_1)$.
Suppose $\phi: B_f(M_1) \to B_f(M_2)$ is a isomorphism
mapping the ray associated to the flat $\{i\}$ to a ray of a corank one flat in $M_2$. Then $M_1 / \{i\}$ is totally disconnected. 
 \end{lemma}

\begin{proof} Let $d+1$ be the rank of $M_1$ and $M_2$.
By  Formula \ref{formula:Eur}, we can compute
$$\deg x_i^d = (-1)^{d-1}\mu^{d-1} (M_1/  \{i\}) \text{ and } \deg x_F^d = (-1)^{d-1}\mu^0(M_2 /F) = (-1)^{d-1},$$
where $\rank F = d$. By Proposition \ref{cor:chowdeg} the isomorphism $\phi$ induces an isomorphism of Chow rings which is compatible with the degree map. Therefore, we must have $\mu^{d-1} (M_1/  \{i\})  = 1$.
It follows from Lemma \ref{lem:equivtotallydisconnected} part 5), that the matroid $M_1/  \{i\}$ is totally disconnected.
\end{proof}

\begin{thm}\label{thm:fineMatroidAuto} Let $M_1$ and $M_2$ be simple matroids such that  neither is totally disconnected. If  $\phi: B_f(M_1) \to B_f(M_2)$
is an isomorphism of fans,  then $\phi$ is induced by a matroid isomorphism. 
\end{thm}

\begin{proof}
Note that if there exists an isomorphism of their Bergman fans, then  $M_1$ and $M_2$ must have the same rank, which we denote by $d+1$. If $M_1$ and $M_2$ are of rank $2$, then the simplicity assumption implies that both are isomorphic to $U_{2, n}$ for some $n$. Hence we may assume that $d \geq 2$. 

Since $\phi$ maps rays to rays, we get a bijective correspondence $\phi$ between the flats of $M_1$ and the flats of $M_2$ preserving adjacency,
i.e.~if $F_1 \subsetneq F_2$ is a flag of flats in $M_1$, we find that either $\phi(F_1) \subsetneq \phi(F_2)$ or $\phi(F_2) \subsetneq \phi(F_1)$.

We begin by  showing that  flats of rank $d$ of $M_1$ must be mapped to flats of rank $d$ of $M_2$. Let $F$ be a flat of rank $d$ in $M_1$. 
As in the proof of Lemma \ref{lem:localtotallydisconnected},
we deduce from Theorem \ref{formula:Eur} that we have 
$\deg(x_{F} ^d) = (-1)^{d-1}$
 in the Chow ring of $M_1$. 
 If $F'$ is a flat of rank $k$ in $M_2$,  find again by Theorem \ref{formula:Eur} that
\[\deg(x_{F'} ^d) = (-1)^{d-1}\binom{d-1}{d-k} \mu^{d-k}(M_2/ F').\] 
By Proposition \ref{cor:chowdeg}, the isomorphism $\phi$ induces an isomorphism of Chow rings compatible with the degree map. 
If $\phi(F)$ is a flat of rank $k$, we deduce that $\deg(x_{\phi(F)} ^d) = (-1)^{d-1}$, which implies $\binom{d-1}{d-k} \mu^{d-k}(M_2/ \phi(F)) = 1$. In particular,  $k$ must be equal to  $d$ or to $1$. 

Assume that $k = 1$, so that $\phi(F) = \{i\}$ for some $i \in E(M_2)$. In this case we find
$\mu^{d-1}(M_2/ \{i\}) =1$, which implies by Lemma \ref{lem:localtotallydisconnected} that the matroid $M_2/\{i\}$ is totally disconnected.

Consider any $j \in F$, then by Theorem \ref{formula:Eur} we have 
$\deg(x_jx_F^{d-1} ) = (-1)^{d-2}$. Suppose that $j$ is sent to the flat $G$ of $M_2$. 
Then we must have $\deg( x_{i}^{d-1}x_G ) = (-1)^{d-2}$. However, the
degree of the product $ x_{i}^{d-1}x_G$ contains 
the factor $\mu^{d-2}(M_2|G/i)$. The rank of the matroid $M_2|G/i $ is $\rank (G) -1$, since $i$ is not a loop. So if $\rank(G) < d$, then  
$\mu^{d-2}(M_2|G/i) = 0$, hence the degree of $ x_{i}^{d-1}x_G$ would be zero, see also \cite[Proposition 3.12]{eur}.
Hence every $j \in F$ must be sent to a flat of rank $d$ of $M_2$. 
Carrying on, we see that all rank $d$ flats of $M_1$ which contain any  $j \in F$ must be sent to rank $1$ flats of $M_2$.

 Now  take any rank $d$ flat $F''$ in $M_1$. We can always find a rank $d$ flat $F'$ in $M_1$ such that $F \cap F'$ and $F' \cap F''$ are non-empty, which implies that $\phi(F'')$ also has to have rank one.  Reasoning as above we find that every rank $d$ flat in $M_1$ is mapped to a rank one flat in $M_2$. Note that conversely every rank one flat in $M_2$ is in fact the image of a rank $d$ flat in $M_1$: Any rank one flat $G$ in $M_2$ is of the form $\phi(F)$ for some flat $F$ in $M_1$. If $F$ were contained in two different rank $d$ flats, then $G$ would be adjacent to two different lines, which is impossible. Therefore  we find by Lemma \ref{lem:localtotallydisconnected} applied to $\phi^{-1}$ that $M_2/\{i\}$ is totally disconnected for all $i \in E(M_2)$ and by Lemma \ref{lem:equivtotallydisconnected} that $M_2$ is totally disconnected which  contradicts our hypothesis. Therefore we conclude that all flats of rank $d$ in $M_1$ are sent to flats of rank $d$ in $M_2$.

Now we show by downward induction that for all $k = 2, \ldots, d$ the map $\phi$ maps flats of rank $k$ to flats of rank $k$. 
Assume that the claim is true for all $l > k$, where $k$ is between $2$ and $d-1$. 
Then $\phi$ induces a map between the set of flats of $M_1$ and $M_2$ of  rank less than or equal to $k$. In particular, it gives a map between the lattice of flats of the truncation matroids $\tr_{k+1}(M_1)$ and $\tr_{k+1}(M_2)$. Recall that the $k+1$-truncation of the matroid $M$ has rank function $$\tr_{k+1}r_M(I ) = \min\{ r(I), k+1\}.$$
Assume that $\phi$ maps a flat of rank $k$ to a flat of rank strictly less than $k$. We apply the argument above to the truncated matroids. Replacing $d$  by $k$ we find that
$\tr_{k+1} M_2 / \{i\} $ is totally disconnected for all $i \in E(M_2)$. So again by Lemma \ref{lem:equivtotallydisconnected}, the truncation  $\tr_{k+1} M_2 $ is totally disconnected and all circuits are of size less than or equal to $2$. 

We claim that this implies that $M_2$ is totally disconnected for $k \geq 2$.  
Suppose that $M_2$ has a circuit $C$ with $|C| \geq 3$. If 
$\tr_{k+1}r_{M_2}(C) = \min\{ r_{M_2}(C), k+1\} = r_{M_2}(C) = |C| -1$, then $C$ is still a circuit  in $\tr_{k+1} M_2$ which is a contradiction. Otherwise we have $|C| > k+2$. 
If $C' \subsetneq C$ and $|C'| = k+2 $, then $C' $ is a circuit of $\tr_{k+1} M_2 $ since $\min\{ r_{M_2}(C'), k+1\}  = |C'| - 1$ and every subset $I \subsetneq C'$ satisfies
$\min\{ r_{M_2}(I), k+1\}  = |I|$. However, $k+2 \geq 3$ which contradicts that $\tr_{k+1} M_2$ is totally disconnected. 
\end{proof}

 \begin{corollary}
Let $\mathcal{W}_1$ and $\mathcal{W}_2$ be the maximal wonderful compactifications of the complements $\Omega_{\mathcal{A}_1}$ and $ \Omega_{\mathcal{A}_2}$ of two essential and connected hyperplane arrangements $\mathcal{A}_1$, $\mathcal{A}_2$. 
Any isomorphism $\Omega_{\mathcal{A}_1} \to  \Omega_{\mathcal{A}_2}$ extending to the maximal wonderful compactifications gives rise to a map of Bergman fans which is induced by an isomorphism of matroids $M(\mathcal{A}_1) \to M(\mathcal{A}_2)$. 
\end{corollary}
\begin{proof} 
As in the proof of  \cite [Theorem 5.1]{kuwe} we see that an isomorphism $f: \Omega_{\mathcal{A}_1} \to  \Omega_{\mathcal{A}_2}$ induces an isomorphism between the intrinsic tori and hence via tropicalization a $\mathbb{Z}$-linear  isomorphism mapping $B(M(\mathcal{A}_1))$ to $B(M(\mathcal{A}_2))$. 

If $f: \Omega_{\mathcal{A}_1} \to  \Omega_{\mathcal{A}_2}$ extends to the maximal wonderful compactifications, it induces an isomorphism between the boundaries. 
The boundary of the maximal wonderful compactification of $\Omega_{\mathcal{A}_i}$ is the union of divisors associated to the flats in the fine structure of the Bergman fan associated to $M(\mathcal{A}_i)$. Since two divisors meet if and only if the associated flats are nested, $f$ induces an isomorphism between the two Bergman fans with their fine structure. Hence our claim follows from Theorem \ref{thm:fineMatroidAuto}.
\end{proof}

\begin{cor} Every finite group $G$ occurs as the automorphism group of the Bergman fan $B_f(M)$ for some simple matroid $M$ of rank $3$. 
\end{cor} 
\begin{proof} This follow from Theorem \ref{thm:fineMatroidAuto} together with the main statement of \cite{boku} which shows   that every finite group is the automorphism group of a simple rank $3$ matroid.
\end{proof}

\section{Isomorphisms of coarse fan structures}\label{sec:isosCoarse}

In this section we investigate isomorphisms between Bergman fans endowed with the coarse structure for loopfree matroids of higher rank.
The following trivial example shows that we need to make some connectedness assumptions to get interesting results.
\begin{example} \label{ex: U} For the totally disconnected matroid $M= U_{n,n}$ on $n \geq 3$ elements the Bergman fan $B(M)$ is simply $\R^n/ \R \mathbbm{1} $. Hence the automorphism group of $B_c(M)$ is isomorphic to $GL_{n-1}(\mathbb{Z})$. 
\end{example}

If Bergman fans decompose in a product of fans, automorphisms can be defined component-wise. We will show now that such a product decomposition only happen for matroids which are non-trivial parallel connections.

\begin{thm}\label{thm:product} 
Let $M$ be a loopfree matroid of rank at least $3$. If $M$ is a non-trivial parallel connection of matroids $M_1$ and $M_2$, then  the Bergman fan  $B_c(M)$
is isomorphic to the product $B_c(M_1) \times B_c(M_2) $. On the other hand, if  $B_c(M)$ is isomorphic to a product $B_c(M_1) \times B_c(M_2) $ of two Bergman fans for matroids $M_i$ of rank at least two, then $M$ is a non-trivial parallel connection.\end{thm} 

\begin{proof}
We begin with the first statement. Consider two matroids $M_1$ and $M_2$  on the ground sets $E_1$ and $E_2$, and suppose that $M$ is the parallel connection of  $M_1$ and $M_2$ along $p$. Hence $E = E(M) = E_1 \cup E_2$ and  $E_1 \cap E_2 = \{p\}$.

If $p$ is a coloop of either $M_1$ or $M_2$, then the parallel connection along $p$ is a disconnected matroid and the first statement follows from \cite[Lemma 2.1]{FrancoisRau}.
Hence we may assume that $p$ is neither a loop nor a coloop in either $M_1$ or $M_2$. We 
 will show that the support of the fan ${B}_c(M)$ can be mapped bijectively to 
$$B_c(M_1) \times  B_c(M_2) \subset \R^{E_1}/ \R \mathbbm{1}_{E_1} \times \R^{E_2}/\R \mathbbm{1}_{E_2}.$$

We use the notation  $w_i$ for the canonical  basis element of  $\Z^{E_1}$ or $\Z^{E_2}$ when $i $ is in $E_1$ or $E_2$. The element $p$ in the intersection $E_1 \cap E_2$ gives rise to vectors $w_{p_{E_1}} \in \R^{E_1}$ and $w_{p_{E_2}} \in \R^{E_2}$.

Now consider the map
\begin{equation}\label{mapParallel}
\phi_M: \R^{E_1 \cup E_2} / \R \mathbbm{1}_{E_1 \cup E_2} \to \R^{E_1}/ \R \mathbbm{1}_{E_1} \times \R^{E_2}/ \R \mathbbm{1}_{E_2}
\end{equation}
given by $v_i \mapsto w_i$ for  $i  \neq p $ and $v_p \mapsto w_{p_{E_1}} + w_{p_{E_2}}.$
This map is an isomorphism of vector spaces, and the inverse is given by $w_i \mapsto v_i$ for $i \in E_1 \cup E_2$, $i \neq p$ and 
$w_{p_{E_i}} \mapsto - v_{E_{i} \backslash \{p\}}.$

Recall that the affine Bergman fan $\tilde{B}(M)$ consists of points $\sum a_iv_i  \in \R^E$ such that $\max_{i \in C} a_i $ is attained at least twice for  all circuits $C$ of $M$. 
Since $p$ is not a loop or coloop in either $M_1$ or $M_2$, we find by
\cite{oxley}, Proposition 7.1.4, that the circuits of the parallel connection $M$ of $M_1$ and $M_2$ along $p$ are of the form $C = C_i$ where $C_i$ is a circuit of $M_i $ or $C = C_1 \backslash \{p\}  \cup C_2 \backslash \{p\}$ where $C_i $  is a circuit of $M_i$ containing $p$.  
A direct check shows that $x \in  B_c(M)$ if and only if $\phi_M(x) \in B_c(M_1 ) \times  B_c(M_2).$

Conversely, suppose that $B_c(M)$ is isomorphic to the product of two Bergman fans $B_c(M_1)$ and $B_c(M_2)$ via a linear isomorphism $\phi: \R^{E} / \R \mathbbm{1}_{E} \to \R^{E_1}/ \R \mathbbm{1}_{E_1} \times \R^{E_2}/ \R \mathbbm{1}_{E_2}$, where 
$M_i$ is a matroid of rank  $d_i + 1$ on a ground set $E_i$ with $|E_i|= n_i$.
Suppose $|E| = n+1$ and $M$ has rank $d$.
Then $d = d_1 + d_2$ and $n+1 = n_1 + n_2 - 1$.
By passing to the simplification of $M$ if necessary, we
may suppose that $M$ has no parallel elements.
Assume that every $i \in E$ gives rise to a ray in the coarse fan structure. 
Since $B_c(M)$ is a product, each of its rays corresponds to a ray of either $B_c(M_1)$ or $B_c(M_2)$. 
This produces a partition of the rank $1$ rays of $B_c(M)$, which equivalently defines a partition of $E = R_1 \sqcup R_2.$ 

Again let $v_k$  denote the standard basis vector of $\R^E$ corresponding to $k \in E$. 
The map $\phi$ sends $\phi(v_i)$ to $\R^{E_1}/\R \mathbbm{1}_{E_1} \times 0$ for $i \in R_1$ and $\phi(v_j)$ to 
$0 \times \R^{E_2}/\R \mathbbm{1}_{E_2}$.
Since $n+1 = n_1 +n_2-1$,  either $|R_1| \geq n_1$ or $|R_2| \geq n_2$,  so  that either the  vectors $\{\phi(v_i) \ | \ i \in R_1\} $
are linearly dependent in $\R^{E_1}/ \R \mathbbm{1}_{E_1}$, or the vectors $\{ \phi(v_j) \ | \ j \in R_2\}$ are linearly dependent in $\R^{E_2}/ \R \mathbbm{1}_{E_2}$. 
However, there is only one linear dependency among the vectors $v_k$ for $ k\in E$, namely  $\sum_{k \in E} v_k = 0$ and the map $\phi$ is an isomorphism. Therefore, our assumption that every $i \in E$ gives rise to a ray in the coarse fan structure is false. 
It follows from Lemma \ref{lem:parallelConnectionMissing}, that the matroid 
$M$ must be a parallel connection along the element $i$. 
\end{proof}

The next result shows that Bergman fans of non-isomorphic parallel connection matroids may be isomorphic. 

\begin{corollary}\label{cor:isoparallel}
Let $M$ and $M'$ be two loopfree matroids on a set $E$ of size $n+1$  obtained as the non-trivial parallel connection of $M_1$ and $M_2$  along possibly different pairs of elements. Then 
there exists an isomorphism from $B_c(M)$ to $B_c(M')$. 
\end{corollary}

\begin{proof}
It suffices to compose $\phi_M$ and $\phi^{-1}_{M'}$ from Equation \ref{mapParallel} for two parallel connections $M$ and $M'$ of the same matroids $M_1$, $M_2$ along different elements. 
\end{proof}

We will now show that for matroids with a certain connectedness property, automorphisms of the coarse Bergman fan can map rays to rank $k$ only to rays of rank or corank $k$.
\begin{thm}\label{cor:rankcorankk}
Suppose $\phi: B_c(M) \to B_c(M') $ is an isomorphism and let $F$ be a flat of 
$M$ of rank $k$ giving rise to a ray in the coarse structure. Suppose neither $M|F$ nor $M/F$ are non-trivial parallel connections, then $\phi$ sends the ray of $B_c(M)$ in direction $F$ to a ray of rank $k$ or corank $k$ in $B_c(M')$. 
\end{thm}

\begin{proof}
By Formula \ref{eqn:stardescription}, for any flat $F$  any matroid $M$, we have 
$$\Star_{v_{F}}(\tilde{B}(M)) = \tilde{B}(M|{F}) \times \tilde{B}(M/{F}).$$
The above fan has a two dimensional lineality space spanned by $\mathbbm{1}_{E}$ and $v_F$. 
Therefore, taking the quotient of the star of  the projective Bergman fan by $v_{F}$ gives, 
$$\frac{\Star_{v_{F}}(B(M))}{\langle v_{F}\rangle} \cong  
B(M|{F}) \times B(M/{F}),$$
where  $\dim B(M|F) = \rank (F) - 1 = k-1$ as well as  $\dim {B}(M/F) = \rank(M)  -\rank (F) -1 = \rank(M) -k.$ 
If neither  $M|F$ nor $M/F$ are non-trivial parallel connections, then by Theorem \ref{thm:product}  neither $B(M|F)$ nor ${B}(M/F)$ splits as a non-trivial product of Bergman fans.

Suppose the ray  generated by $v_F$ is a ray in the coarse structure, which is mapped by $\phi$ to the ray generated by $v_{F'}$ for a flat $F'$ of $M'$. Then
 the above quotient of the star fan of $B_c(M)$ at $v_F$ must be sent to the 
corresponding quotient of the star fan of $B_c(M') $ at $v_{F'}$. Hence the star of $B(M')$ at $v_{F'}$ must be a product of Bergman fans of the same dimensions.  Hence $F'$ is of rank or corank $k$. 
\end{proof}

In Section \ref{sec:modular}, we will apply this result in the following situation.

\begin{cor}\label{cor:rankcorankone}
Let $\phi$ be an automorphism of the coarse Bergman fan $B_c(M)$ of a loopfree matroid $M$ with ground set $E$. Suppose that for every element of the ground set $i \in E$ the matroid $M/i$ is not a non-trivial parallel connection. Then $\phi$ maps rays of rank $1$ to rays of rank $1$ or of corank $1$. 
\end{cor}
 
 \begin{proof}
 The statement is an immediate specialisation of Theorem \ref{cor:rankcorankk} in the case $k =1$. 
 \end{proof}

\section{Cremona maps}

Let $M$ be a matroid of rank $d+1$ on a ground set $E$ of $n+1$ elements. 
As in Section \ref{sec:MatFans}, we write $v_F = \sum_{i \in F} v_i$ or every subset $F \subseteq E$.  Recall that $\mathbbm{1} = v_E$. 

\begin{definition}\label{def:Cremona}
Consider a basis  $b = \{b_0, \dots, b_{d}\}$ of a simple matroid $M$.
For each basis element $b_j$ in $b$ we define a flat $B_j$ by $$B_j := \cl \{b_0, \dots , \widehat{b_j}, \dots , b_{d}\},$$
where $\widehat{b_j}$ indicates that $b_j$ is left out. 

We define the $\Z$-linear map $\crem_{b} \colon \R^{E}   \to \R^{E} $  by  
$v_{b_j}   \mapsto v_{B_j} $ for the basis elements $b_j$, and 
$v_k  \mapsto v_k $ for all $k \in E \backslash b$. 

If $\crem_b$ maps the line $\R \mathbbm{1}$ to itself, we also write 
$\crem_b: \R^{E } / \R \mathbbm{1} {\rightarrow} \R^{E} / \R \mathbbm{1}$ for the quotient map.

\end{definition}

We wish to know under what conditions on the matroid $M$ the map $\crem_{b}$ preserves the coarse Bergman fan  $B_c(M)$. This is the case if and only if  $\crem_b$ preserves the support  of the (projective) Bergman fan, which we denote by $B(M)$.

Recall the Cremona map $\crem_{E} \colon B(U_{n+1, n+1}) \to B(U_{n+1, n+1})$ from Example \ref{ex:cremUniform}. This map is defined by  $v_i \to -v_i$ for all $i \in E$.
We can consider the action of this map $\crem_{E}$ restricted to Bergman fans $B(M)$ for any matroid $M$ on the ground set $E$. 

\begin{lemma}\label{CremonaOnRn}
Let $M$ be a matroid of rank $d+1$ on the ground set $E$. The image of $B(M)$   under $\crem_{E}$ is the support of a Bergman fan $B(M')$  if and only if $M$ is totally disconnected.
\end{lemma}

\begin{proof}
Suppose $|E| = n+1$. 
We use the description of the reduced characteristic polynomial of Huh and Katz \cite{HuhKatz} in terms of intersection theory of Minkowski weights on the $n$-dimensional permutahedral toric variety. This is the toric variety defined by the complete fan $B(U_{n+1, n+1})$.  The map $\crem_{E}$ is an automorphism of the fan $B(U_{n+1, n+1})$ and hence produces an automorphism of its Chow ring compatible with the degree map by Proposition \ref{cor:chowdeg}. Hence we also have an automorphism of the ring of Minkowski weights.

  By \cite[Lemma 6.1]{HuhKatz}, there is a Minkowski weight $\alpha$ such that for all $i$, we have 
 $\mu^i (M) =  \alpha^{d-i} \cup \crem_{E}(\alpha)^i \cup B(M)$, 
 where $\cup $ denotes the intersection product and $\mu^i(M)$ is the $i$-th unsigned coefficient of the reduced characteristic polynomial of $M$.
Since the map $\crem_{E}$ gives an automorphism of the ring of Minkowski weights we have  
$$\alpha^d \cup B(M) =  \crem_{E}(\alpha)^d \cup B(M).$$
Now the left hand side is $\mu^0(M)   = 1$ and the right hand side is equal to  $\mu^d (M)$. 
By Lemma \ref{lem:equivtotallydisconnected}, the coefficient $\mu^d(M)$ is equal to one if and only if $M$ is totally disconnected.  \end{proof}

We will now prove a criterion about the existence of Cremona maps (with respect to a suitable basis) acting as automorphisms of the Bergman fan with the coarse structure. 

\begin{thm}\label{thm:cremiff} Let $b$ be a basis of a simple connected matroid $M$.  For any pair of elements $i, j \in E$, let $F_{i j}$ be the rank $2$ flat which is the closure of $i$ and $j$. The map $\crem_b$ descends to a linear map $\crem_b:  \R^{E} / \R \mathbbm{1} \rightarrow \R^{E} / \R \mathbbm{1}$ mapping  $B(M)$  to itself, if only if the sets 
$\{F_{i j} \backslash \{i, j\}\}_{i, j \in b}$ partition the set $E \backslash b$ into pairwise disjoint subsets.  
\end{thm}

We will need the following lemma in the proof. 

\begin{lemma}\label{lem:CremInvolution} 
Assume that $b$ is a basis of a simple matroid $M$, such that the sets $\{F_{i j} \backslash \{i, j\}\}_{i, j \in b}$ partition $E \backslash b$. Then $\crem_b$ descends to a linear automorphism on the quotient $\R^E / \R  \mathbbm{1}$.

For all subsets $a \subset b$ we define an associated flat $F_{a} = \cl\{ b_i: b_i \in a\}$. 
Then each flat $G$ of $M$ decomposes into the disjoint union of the two flats 
$F_{b \cap G}$ and $G \cap F_{b \backslash G}$. Moreover, 
$$\crem_b(v_{G \cap F_{b \backslash G}})  =v_{G \cap F_{b \backslash G}}  \qquad \text{and}  \qquad \crem_b(v_{F_{b \cap G}})  = v_{F_{b \backslash G}}  \in \R^{E} / \R \mathbbm{1}.$$
\end{lemma}

\begin{proof}
Note that the partition hypothesis implies that every element $a \in E \backslash b$ is contained in precisely $d-1$ of the sets $B_j := \cl \{b_0, \dots , \hat{b_j}, \dots , b_{d}\}$. Therefore the Cremona automorphism $\crem_b$ of $\R^{E}$ maps 
$\mathbbm{1}$ to $\sum_{i=0}^d v_{B_i} + \sum_{a \notin b} v_a = d \mathbbm{1}$, hence it descends
to the quotient space $\R^{E} / \R \mathbbm{1}$ as the endomorphism  $\crem_b$ from Definition \ref{def:Cremona}.

Let $G$ be any flat in $M$ and consider the subflat $F_{b \cap G} \subset G$. Let $h \in G \backslash F_{b \cap G}$. By assumption, $h \in \cl\{b_i,b_j\} $ for some $i,j$. The basis vector $b_i$ cannot be contained in $G$, since this would imply $b_j \in \cl\{h,b_i\} \subset G$, and therefore $h \in F_{b \cap G}$. For the same reason the basis vector $b_j$ cannot be contained in $G$. Hence $h \in F_{b \backslash G}$. This implies that $G \backslash F_{b \cap G} = G \cap F_{b \backslash G}$. Therefore $G$ is the disjoint union of the two flats $F_{b \cap G}$ and $G \cap F_{b \backslash G}$. 

To prove the remaining statements, note that $G \cap F_{b \backslash G} $ does not intersect $b$. Therefore, the map $\crem_b$ fixes the vector $v_{G \cap F_{b \backslash G}}$. 
 Next put  $g= |b \cap G|$ and consider the action of Cremona:
\[\crem_b(v_{F_{b \cap G}}) = \sum_{b_i \in b \cap G} \crem_b(v_{b_i})  + v_{F_{b \cap G} \backslash b} = \sum_{b_i \in b \cap G} v_{B_i}  + v_{F_{b \cap G}\backslash b}.  \]
We write $C_{pq} = F_{b_p,b_q} \backslash \{b_p,b_q\}$. By our assumption on the matroid $M$ the set  $F_{b \cap G} \backslash b $ is the union of all $C_{pq}$ for $b_p$ and $b_q$ in $b \cap G$, so that $v_{F_{b \cap G} \backslash b} = \sum_{b_p, b_q \in b \cap G} v_{C_{pq}}$. 
Similarly, $B_i \backslash b$ is the union of all $C_{pq}$ for all  $b_p$ and $b_q$ both not equal to $b_i$, so that  $v_{B_i \backslash b} = \sum_{b_p \neq b_i, b_q \neq b_i} v_{C_{pq}}$.

Hence we can calculate
\[\sum_{b_i \in b \cap G} v_{B_i} = 
 g \sum_{b_j \in b \backslash G} v_{b_j} + (g-1) \sum_{b_j \in b \cap G} v_{b_j} +\]
 \[  (g-2) \sum_{b_p,b_q \in G} v_{C_{pq}} + (g-1) \sum_{\mbox{either }b_p \mbox { or } b_q \in G} v_{C_{pq}} + g \sum_{b_p, b_q \notin G} v_{C_{pq}},\]
 so that
 \[\crem(v_{F_{b \cap G}}) =  v_{F_{b \backslash G}} +  (g-1) \mathbbm{1} .\]
 This proves the lemma. 
\end{proof}

\begin{proof}[Proof of Theorem \ref{thm:cremiff} ]

Let $b$ be a basis such that $\crem_b$ descends to $\R^E / \R \mathbbm{1}$ and preserves $B(M)$.  Put $C = E \backslash b$. We consider the matroid $M|b$ which is the same as the matroid $M \backslash C$ and isomorphic to  the uniform matroid $U_{d+1, d+1}$ for $d+1 = rk(M)$. For each $k \in C$, consider the rank $d+1$ matroid
$M \backslash \{ C \backslash k\}$.
This matroid is loopfree since $M$ is assumed to be loopfree. Therefore, $M \backslash \{ C \backslash k\} / k$ is a 
matroid of rank $d$
 on the ground set $b$.
 
Since $M$ is simple, an element $i\in b$ and $k$ cannot be parallel in $M$, so that $i$ is not a loop of $M \backslash \{ C \backslash k\} / k$,
and we can consider its Bergman fan in $\R^{b}/ \mathbbm{1}_{b}$. It
 must be preserved under $\crem_b$ acting on $\R^{b}/ \mathbbm{1}_{b}$,
since the Cremona map commutes with coordinate projections and the Bergman fans of matroid minors of $M$ can be determined from the coordinate projections \cite[Proposition 2.22]{Shaw}. 

By Lemma \ref{CremonaOnRn}, the matroid $M / k \backslash \{ C \backslash k\}$ must be totally disconnected for all $k$. 
Since it is a rank $d$ loopfree matroid on the $d+1$ elements of $b$, by the pigeon hole principle there must exist a unique pair of elements in $b$, say $\{i, j\}$, which are parallel elements  in $M / k \backslash \{ C \backslash k\}$. This implies that $k \in F_{i,j}\backslash \{i,j\}$ and we have a partition of $E \backslash b$ given by the sets $F_{i, j} \backslash \{i, j\}$ for $i, j \in b$.

For the other direction assume  that $b$ is a basis of $M$ such that the sets $\{F_{i j} \backslash \{i, j\}\}_{i, j \in b}$ partition $E \backslash b$. For simplicity we write again $C_{ij} :=F_{ij}\backslash \{i,j\}$. Then $C = \bigcup_{p,q} C_{pq}$. 

Note that if  $k \in E \backslash b$, then $v_k$ is fixed by $\crem_b$. If $b_i \in b$, we have 
\begin{eqnarray*}
\crem_b^2(v_{b_i}) 
& = &  \crem_b(v_{B_i}) =  \sum_{j \neq i} v_{B_j} + \sum_{p \neq i, q \neq i} v_{C_{pq}} \\& = & d v_{b_i} + (d-1) \sum_{j \neq i} v_{b_j} + (d-1)\sum_{k \in C} v_k,
\end{eqnarray*}
so that under the above hypothesis $\crem_b^2(v_{b_i})= v_{b_i} + (d-1) \mathbbm{1}.$ Therefore, 
$\crem_b$ is an involution on $\R^{E} / \R \mathbbm{1}$. 

We have to show that $\crem_b$ preserves the support $B(M)$ of the Bergman fan.
 Since we are only interested in the support of the fan, the fan structure we consider on it does not matter. Therefore it suffices to show that for every maximal cone $\gamma$ in $B_f(M)$ (with respect to the the fine fan structure), the image $\crem_b(\gamma)$ is contained in a cone in $B_m(M)$ (with respect to the minimal nested set structure). Namely this implies
 $\crem_b(B(M) \subset {B}(M)$,   and since $\crem_b$ is an involution, equality of sets follows.

Again write $F_{a} = \cl\{ b_i: b_i \in a\}$ for all subsets $a$ of our basis $b$. Hence $F_{ij} = F_{\{b_i,b_j\}}$.
 Let $\mathcal{G}:= \emptyset \subsetneq G_1 \subsetneq \ldots \subsetneq G_{d}$ be a flag
 of pairwise different flats in $E$, giving rise to the maximal cone $\rho_{\mathcal{G}} := \langle v_{G_1}, v_{G_2} , \dots , v_{G_d} \rangle_{\mathbb{R}_{\geq 0}}$ in the fine subdivision of $B(M)$.

 For every $j \in \{1, \ldots, r\}$ we write $G_j$ as the union of the two flats $F_{b \cap G_j}$ and $G_j \cap F_{b \backslash G_j}$ as in the proof of Lemma \ref{lem:CremInvolution}.  The sets $b \backslash G_j$ give a chain of pairwise different subsets $\emptyset \subsetneq I_1 \subsetneq ...\subsetneq I_r $ of $b$. 
Moreover, we decompose every flat $G_j \cap F_{b \backslash G_j}$ into its  connected components. All those components for all $G_j$ form a finite set of connected flats $K_1, \ldots, K_s$ of $M$, such that each $K_i$ are contained in $E\backslash b$. 
 
 We claim that the set of flats $\mathcal {F} = \{F_{I_1}, \ldots, F_{I_r}, K_1, \ldots, K_s\}$ is nested for the minimal nested set structure of $M$. Consider a subset $\mathcal {S}$ of $\mathcal{F}$  consisting of at least two pairwise incomparable elements. This implies that at most one of those elements is of the form $F_{I_k}$. The other elements of $\mathcal{S}$ are given by pairwise incomparable flats $K_i$ for $i \in I$, where $I$ is a subset of $\{1, \ldots, s\}$.  We have to show that the flat 
 $\cl(\bigcup_{S \in \mathcal{S}} S)$ is disconnected. 
 
 Assume first  that $\mathcal{S}$ only contains  flats of the form $K_i$. Each $K_i$ is contained in some $G_{j(i)}$. We can find some $G_j$ containing all of them, so that at least one of the $K_i$ is contained in $G_j \cap F_{b \backslash G_j}$.  Since $G_j$ decomposes into the two flats $F_{b \cap G_j}$ and $G_j \cap F_{b \backslash G_j}$, we find  
  \[\cl\{S: S \in \mathcal{S}\} =  ( \cl\{S: S \in \mathcal{S}\}  \cap F_{b \cap G_j}) \cup  (\cl\{S: S \in \mathcal{S}\} \cap G_j \cap F_{b \backslash G_j})\]
 If at least one of the elements $K_i$ in $\mathcal{S}$ is contained in the first set, then this decomposition shows that the flat $\cl\{S: S \in \mathcal{S}\}$ is disconnected. If all $K_i$ in $\mathcal{S}$ are contained in $G_j \cap F_{b \backslash G_j}$ which is the disjoint union of some of the flats $K_q$ by construction, we
 find that $ \cl\{S: S \in \mathcal{S}\}$ is a disconnected subflat of  $G_j \cap F_{b \backslash G_j}$.

 Let us now assume that $\mathcal{S}$ contains $F_{I_k}$, and let $G_j$ be an element in our flat $\mathcal{G}$ such that $I_k = b \backslash G_j$.
 For every $K_i$ in $\mathcal{S}$ there exists some $G_{j(i)}$ containing $K_i$, i.e.~$K_i \subset G_{j(i)} \cap F_{b \backslash G_{j(i)}}$. Since $K_i$ and $F_{I_k}$ are incomparable, this implies that $j > j(i)$. Hence all $K_i$ are contained in the flat $G_j $, which is the disjoint union of $F_{b \cap G_j}$ and $G_j \cap F_{b \backslash G_j}$. Since $K_i$ and $F_{G_j \backslash b}$ are disjoint, we deduce that all $K_i$ are contained in fact in $F_{b \cap G_j}$. 
 Therefore we can decompose
 \[\cl\{S: S \in \mathcal{S}\} =  ( \cl\{S: S \in \mathcal{S}\}  \cap F_{b \cap G_j}) \cup  (\cl\{S: S \in \mathcal{S}\} \cap F_{b \backslash G_j})\]
 into two disjoint non-empty flats, so that the flat on the left hand side is indeed disconnected. 
 
This concludes the proof that the set of flats  $\mathcal{F}= \{F_{I_1}, \ldots, F_{I_r}, K_1, \ldots, K_s\}$ is nested for the minimal nested set structure. Hence the cone
  \[\rho_{\mathcal{F}} = \langle v_{F_{I_1}}, \dots , v_{F_{I_r}}, v_{K_1}, \ldots v_{K_r}  \rangle_{\mathbb{R}_{\geq 0}} + \R \mathbbm{1} \]
   is a cone in the minimal nested set structure of the Bergman fan.  
  
By Lemma \ref{lem:CremInvolution},   we have $v_{G_j} = v_{F_{b \cap G_j}} + v_{G_j \cap F_{b \backslash G_j}}$, and $$\crem_b(v_{G_j} )= v_{F_{b \backslash G_j}} + v_{G_j \cap F_{b \backslash G_j}} + \R \mathbbm{1}.$$ Hence we deduce that
  the image of the cone $\rho_{\mathcal{G}}$ under $\crem_b$ is contained in $\rho_{\mathcal{F}}$, and hence in $B(M)$. 
\end{proof}

\section{The Cremona group of a matroid in rank $3$}\label{sec:rank3}

For matroids of rank $3$, we have the following strong result:

\begin{thm}\label{prop:rank3}\cite[Theorem 2.8]{Shaw:Surfaces}
If  $M_1$ and $M_2$ are simple matroids of rank $3$ and there exists an isomorphism  $\phi: B_c(M_1)  \to B_c(M_2)$, then $M_1$ and $M_2$ are isomorphic.
\end{thm}

The statement of the above proposition asserts that $M_1$ and $M_2$ are isomorphic, however this does not mean that $\phi$ is necessarily induced by a matroid isomorphism. We have seen in the last section, that there are Cremona automorphisms preserving Bergman fans and these do not come from matroid automorphism. A concrete example to keep in mind is the matroid $M$ of the braid arrangement $A_3$, see Example \ref{ex:braid}. In this case, we have  $\text{Aut}(M) \subsetneq \text{Aut}(B(M)) = S_5$. 

We will prove Theorem \ref{prop:rank3} by considering the case of parallel connections separately. 

\begin{lemma}\label{lem:parallel3}
Let $M_1$ and $M_2$ be simple matroids of rank $3$. If one of them is a non-trivial parallel connection, and there exists an isomorphism  $\phi: B_c(M_1) \to B_c(M_2)$, then $M_1$ and $M_2$ are isomorphic. 

Moreover the automorphism group of $B_c(M)$ when $M$ is a non-trivial parallel connection of rank $3$ is the automorphism group of a complete bipartite graph.
\end{lemma}

\begin{proof}
Firstly, if one of the matroids $M_1$, $M_2$  is a non-trivial parallel connection and there exists an isomorphism of their Bergman fans as in our claim, then both matroids are non-trivial parallel connections by Proposition \ref{thm:product}.
Since $B_c(M_i)$  is two dimensional, the matroids $M_i$ are non-trivial  parallel connections if and only if their fans are a product of two  one dimensional projective Bergman fans. The number of rays and faces of the fans determines the number of rays in each of these one dimensional Bergman fans, which in turn determines the two rank $2$ matroids involved in the parallel connections.  Hence the matroids are isomorphic. 

Since the fan  $B_c(M)$ of a rank $3$ parallel connection of two rank $2$ matroids is a product of two one dimensional fans, the fan is the cone over a complete bipartite graph. Hence any automorphism of the fan induces an automorphism of the complete bipartite graph. Conversely, any automorphism of this complete bipartite graph yields an integer linear map which preserves the Bergman fan of the matroid.
  \end{proof}

\begin{thm}\label{thm:cremona3} Let 
 $M_1$ and $M_2$ be simple matroids of rank $3$ which are not non-trivial parallel connections.  Any  isomorphism  $\phi: B_c(M_1) \to B_c(M_2)$  is either induced by a matroid isomorphism or the composition of a Cremona map and matroid isomorphism. 

\end{thm}

To prove the above lemma we will require some calculations in the Chow ring of the coarse fan structure on $B(M_i)$. Note that for a rank $3$ matroid $M$ which is not a non-trivial parallel connection, the coarse fan structure and the minimal nested set structure of the Bergman fan coincide by Lemma \ref{lem:coarsemin}.

\begin{lemma}\label{lem:coarseChow3} 

Let $A_c^*(M)$ denote the Chow ring of the coarse subdivision of the fan of a rank $3$ simple matroid $M$ which is not a parallel connection of two rank $2$ matroids. 
Let $F, F'$ be flats of rank $2$ of  $M$ which produce rays in 
$B_c(M)$. Let $\deg $ denote the isomorphism 
$A^2_c(M) \to \Z$. Then 
\begin{enumerate}
\item $\deg x_k x_F = 1$ if $k \in F$,  for all $k \in E$, 

\item $\deg x_k^2 = 1 - |\{F \ | \ k \in F,  |F | > \rank F = 2\}|$,
\item  $\deg x_F^2 = -1$,
\item   $\deg x_Fx_F' = 0$ for $F \neq F'$. 
\end{enumerate}
\end{lemma}

\begin{proof}
To simplify notation let $\Sigma$ denote $B_c(M)$. 
We recall that the degree map is defined on $A^2_c(M) $ by $\deg x_k x_F = 1$ if $k \in F$ \cite{AminiPiquerez}. 
Recall from Section \ref{sec:Chowrings} that 
\begin{equation}\label{equ:coarseChow} 
 \sum_{k \in F} x_F - \sum_{l \in G} x_G \in J, 
\end{equation}
where $F$ and $G$ denote connected flats of rank $2$.
 Thus this sum is $0$ in $A^1_c(M)$.  
Therefore, for any $l \in E$ not equal to $k$ we have 
\begin{equation}\label{equ:coarsexk^2} 
x_k^2 = x_k (x_l + \sum_{\substack{l \in G \\ G \neq l}} x_G -\sum_{\substack{ k \in F \\ F \neq k}} x_F ).
\end{equation}

There are two cases to consider, either $\cl\{l,k\} = \{l,k\}$ or not. 
If $\cl\{l,k\} = \{l,k\}$, then the rays $\rho_l$, $\rho_k$, corresponding to $l$ and $k$ respectively, span a face of $\Sigma$ and there is no face spanned by $\rho_k$ and $\rho_G$ for rank $2$  connected flats $G$ with $l \in G$. Then the product becomes 
$$\deg x_k^{2} = 1 - |\{F \ | \ k \in F,  |F | > \rank F = 2\}|,  $$  as claimed. If $F_{lk} = \cl\{l,k\} \supsetneq \{l,k\}$, then there is a ray for $F$ in the coarse subdivision. Moreover, we have the products $\deg x_l x_k = 0$ and  $\deg x_k x_{F_{lk}} = 1$. Therefore the product in (\ref{equ:coarsexk^2}) still  has degree $$1 -  |\{F \ | \ k \in F,  |F | > \rank F = 2\}|.$$ 
Using the relation (\ref{equ:coarseChow}), we obtain
$$x_F^2 = x_F[x_l + \sum_{\{l\} \subsetneq  G} x_G - ( x_k + \sum_{k \in F', F' \neq F} x_{F'}) ],$$
without loss of generality we can assume that $l \not \in F$, in which case, $x_F^2 = -x_Fx_k$ which has degree $-1$. 

Lastly, if  $F \neq F'$ and both are of rank $2$, then $F$ and $F'$ are not nested hence their corresponding rays do not span a cone of $B_c(M)$. This implies that $\deg x_Fx_F' = 0$. 
\end{proof}

\begin{proof}[Proof of Theorem \ref{thm:cremona3}] 
Recall from Example \ref{ex:mu1}, a simple matroid $M$ with ground set $E$ has $\mu_1(M) = \dim \mathcal{F}_1(M) = |E| -1$,  and the dimension of this vector space is preserved by a fan isomorphism by Proposition \ref{prop:samecharpoly}. Hence, if  $E_1$ and $E_2$ are the ground sets of $M_1$ and $M_2$, respectively, then we must have $|E_1| = |E_2|$. 
By Lemma \ref{lem:parallelConnectionMissing}, every ray of rank $1$ is in $B_c(M_i)$.
Any connected flat of rank $2$ of $M_i$ induces a ray of $B_m(M_i)$ and hence of  $B_c(M_i)$.
 If $P_i$ is the set of connected flats of $M_i$ then we must also have $|P_1| = |P_2|$.

Now if $\phi$ is not induced by a matroid automorphism, then there is a ray of rank $1$ in $B_c(M_1)$ which is sent to a ray of rank $2$ in $B_c(M_2)$ by $\phi$. Let $\tilde{E}_i, \tilde{P}_i$ be the subsets of the ground set $E_i$, and rank two connected flats $P_i$, respectively, whose ranks are changed by $\phi$ or $\phi^{-1}$. 
In particular, the map 
$\phi$ also induces bijections between $\tilde{E}_i$ and $\tilde{P}_j$ for $\{i, j\} = \{1, 2\}$. Moreover, since the number of rank one flats sent to rank one flats is the same for $\phi$ and $\phi^{-1}$,
we have
$|E_1| - |\tilde{E}_1| = |E_2| - |\tilde{E}_2|$. Therefore, all of $\tilde{E}_1,  \tilde{E}_2, \tilde{P}_1,$ and $\tilde{P}_2$ have  the same size.

Let $G$ be the graph with vertex set $V = \tilde{E}_1 \cup \tilde{P}_1$ and with edges between two vertices if and only if the corresponding rays span a face of $B_c(M_1)$. We will prove that this is a bipartite cycle graph with $6$ vertices and with vertex partition given by $V = \tilde{E}_1 \cup \tilde{P}_1$.

Claim $1$: If $k \in \tilde{E}_1$ then every connected rank $2$ flat $F \ni k$ must be in $\tilde{P}_1$. Otherwise if $\phi_ \ast(x_k) = x_{F'} $ and $\phi_ \ast(x_F) = x_{F''}$ then $\deg(x_k x_F) = 1$ yet $\deg(x_{F'}x_{F''}) = 0$ by Lemma \ref{lem:coarseChow3}. This contradicts that $\phi$ induces a ring isomorphism $A_c^*(M_1) \to A_c^*(M_2)$ compatible with the degree map. 

Claim $2$: If $k \in \tilde{E}_1$ then $|\{F \ | \ k \in F,  |F | > \rank F = 2\}| = 2$. By Proposition \ref{cor:chowdeg} the map  $\phi $ induces a Chow ring isomorphism compatible with the degree map. If $\phi_{\ast}(x_k) = x_{F'}$ then $\deg(x_k^2) = \deg(x_{F'}^2 ) = -1$ and by Lemma  \ref{lem:coarseChow3} we have $|\{F \ | \ k \in F,  |F | > \rank F = 2\}| = 2$. 

Claim 1 and 2 together imply that every vertex in $\tilde{E}_1$ has valency $2$ in $G$. Repeating the argument for the inverse map of $\phi$ proves the same assertion for $\tilde{P}_1$. 
Notice that this also implies that $|\tilde{P}_1| \geq 2$ and hence that $|\tilde{E}_1 | \geq 2$. 

Claim $3$: If $k, l \in \tilde{E}_1$ then $F = \cl{\{k,l\}} \in \tilde{P}_1$. In particular, $\cl{\{k,l\}}$  must be a ray in the minimal subdivision. This proves that the graph $G$ is bipartite and connected. 

Therefore, the graph $G$ is a connected 2-regular bipartite graph with $2m$ vertices, where $m = |\tilde{E}_1| = |\tilde{P}_1|$. Hence it must be the cycle graph $C_{2m}$.
If the graph is $C_{2m}$, then there are $m$ elements in $\tilde{E}_1$ and by Claim $3$ a vertex in $\tilde{E}_1$ would have valency $m-1$. Since the graph is $2$-regular we have $m = 3$. 
Moreover, this shows that the map $\phi$ sends $j, k, l$ to rank $2$ flats of $M_2$ 
where $\tilde{E}_1 = \{j,k,l\}$.  All other elements of $E_1$ are sent to elements of $E_2$.

First notice that $\rank(\tilde{E}_1) = 3$, therefore this set is a basis. We prove that the Cremona map $\crem_{\tilde{E}_1}$ preserves the fan by using Theorem \ref{thm:cremiff}.

To apply the theorem we must show that every $p \in E_1 \backslash \{j,k,l\}$ is contained in one of the flats $\cl{\{k,l\}}, \cl{\{j,l\}}, \cl{\{k,j\}}$. 
Suppose otherwise, 
then the ray corresponding to $p \in E_1$ spans two dimensional cones with the rays of $j$, $k$, and $l$. This is because the closure of $p$ and any one of $j, k, $ or $l$ must be disconnected, as  there are only  two  connected rank $2$  flats containing any one of $j,k,l$ and they are among  $\cl{\{k,l\}}, \cl{\{j,l\}}, \cl{\{k,j\}}$. 

By Lemma \ref{lem:coarseChow3}, for any singleton $i$ in a rank $3$ matroid $M$ we have  in $A_c(M)$, $$\deg(x_i^2) = 1 -  |\{ F \ | \ F \text{ flat of }  M, F \ni i,  |F| > \rank F =  2  \}|.$$ 
 Yet since $p \not \in \tilde{E}_1$ its ray is sent to a rank one ray in the fan of $M_2$ corresponding to $p' \in E_2$.
However, the elements $j,k,l$ are sent to connected rank $2$ flats of $M_2$
containing $p'$, so that there are three more rank two flats contributing to $\deg(x_{p'}^2)$ than to $\deg(x_{p}^2)$,
which contradicts Proposition \ref{cor:chowdeg}.

Hence every $p \in E_1 \backslash \{j,k,l\}$ is contained in one of $\cl{\{k,l\}}, \cl{\{j,l\}}, \cl{\{k,j\}}$ and the  Cremona map $\crem_{\tilde{E}_1}$ preserves  the fan. Moreover, the composition of $\crem_{\tilde{E}_1}$ and $\phi$ is sends all rank 1 rays to rank 1 rays and hence is induced by a matroid automorphism. Therefore, the map $\phi$ is the composition of a Cremona map and a matroid automorphism. 
\end{proof}

\begin{cor}\label{cor:autgroup_rank3}
The automorphism group of $B_c(M)$ for a simple rank $3$ matroid $M$ which is not a parallel connection is generated by matroid automorphisms and Cremona maps. 
\end{cor}

\begin{proof}[Proof of Theorem  \ref{prop:rank3}]
The statement follows immediately from Theorem \ref{thm:cremona3}.
\end{proof}
  
 It is an interesting question under which hypotheses an analogous statement is true for matroids of higher rank. In Section \ref{sec:modular}, we prove an analogous statement for modularly complemented matroids.
 
 \begin{example}
  Consider $M$ the parallel connection of $U_{2, n} $ and $U_{2, m}$ along a common  element $0$ contained in both ground sets.  
 By Theorem \ref{thm:product}, this  rank $3$ matroid 
 is the product of two $1$-dimensional Bergman fans. Therefore the automorphism group of $B_c(M)$ is isomorphic to the automorphism group of a complete bipartite graph $K_{m,n}$ where $m$ and $n$ are the number of rays of the two $1$-dimensional fans in the product. We assume that $m$ and $n$ are at least $3$. 
   By Lemma  \ref{lem:parallelConnectionMissing}, the element  $0 \in E$ is such that $v_0$ is not a ray of the coarse fan structure. 
   
   The parallel connection $M$ has  exactly two connected flats of  rank $2$. These are the only connected rank $2$ flats containing $0$.  These produce two rays in the coarse fan structure $B_c(M)$ of rank $2$.   We claim that an automorphism of  $K_{m,n}$ 
 which fixes the first vertex set, yet acts freely on the other vertex set cannot be obtained by composing matroid automorphisms and Cremona maps. Notice that such an isomorphism necessarily fixes one of the two connected rank $2$ rays and sends the other rank $2$ ray to a ray of rank $1$. 
A matroid isomorphism preserves the ranks of rays, thus such an isomorphism either leaves the two rays of rank $2$ fixed or it may  swap them if we are in the case $n =m$.
Any basis $b$ of $M$ containing $0$ gives  a Cremona map $\crem_b$ which is an automorphism of  the coarse Bergman fan of $M$ and hence of $K_{m,n}$. These are the only bases for which Cremona maps are possible. 
Applying the Cremona map swaps the two rays of rank $2$ for two rays of rank $1$. Thus it is not possible to obtain all of $\mathrm{Aut}(K_{m, n})$ from Cremona maps and matroid isomorphisms. 
 \end{example}

\section{Modularly complemented matroids}\label{sec:modular}

For matroids $M_1$ and $M_2$ of rank bigger than $3$, the determination of the fan isomorphisms between $B_c(M_1)$ and $B_c(M_2)$  is much more involved. In the present section we will determine the automorphism groups of $B_c(M)$ for modularly complemented matroids of rank at least $4$.

Let us recall some definitions from \cite{oxley}, Section 6.9. Two flats $X$ and $Y$ in a matroid $M$ form a \emph{modular pair} if 
\[r(X) + r(Y) = r(X \cup Y) + r(X \cap Y).\]
A flat $X$ is called modular, if and only if for all flats $Y$ the pair $(X,Y)$ is a modular pair. A matroid is \emph{modular}, if every flat is modular. 
This is a very restrictive condition. In fact, a  connected simple modular matroid is either a free matroid $U_{n,n}$ or a finite projective geometry \cite{oxley}, Proposition 6.9.1.
Recall from Example \ref{ex: U} that the coarse Bergman fan of the totally disconnected matroid $U_{n,n}$ is a linear space with automorphism group $GL_{n-1} (\mathbb{Z})$. 

We will now deal with the case that $M$ is a finite projective geometry of rank at least $3$, i.e.~$M = PG(n,q)$ for some $n \geq 2$ is the matroid on the set $E(M)$ of all one-dimensional subspaces of the vector space $\mathbb{F}_q^{n+1}$, with rank function given as the dimension of the linear span of a subset of $E$.  Hence the flats in $M$ correspond to the linear subspaces of $\mathbb{F}_q^{n+1}$. The matroid $M = PG(n,q)$ is isomorphic to the matroid represented by the hyperplane arrangement $\mathcal{A}$ in $\mathbb{P}_{\mathbb{F}_q}^n$ consisting of all $\mathbb{F}_q$-rational hyperplanes. It is shown in \cite{rtw} that every automorphism of the complement $\Omega_{\mathcal{A}}$ extends to $\mathbb{P}_{\mathbb{F}_q}^n$, so that the automorphism group of $\Omega_{\mathcal{A}}$ is $PGL(n+1, \mathbb{F}_q)$. 

\begin{proposition} \label{prop:projectivespace}
Let $M = PG(n,q)$ be a finite projective geometry with $n\geq 2$. Then every  isomorphism of $B_c(M)$ is induced by a matroid automorphism.

\end{proposition}

\begin{proof} 
Note that for all connected non-nested flats $F_1$ and $F_2$ in $M$, the closure of $F_1 \cup F_2$ is strictly bigger than $F_1 \cup F_2$, hence it is not disconnected. Therefore the minimal nested set structure and the fine structure on the Bergman fan coincide. 
Moreover, by Lemma \ref{lem:coarsemin} the minimal nested set structure is the coarsest fan structure on $B(M)$. Hence our claim follows from Theorem \ref{thm:fineMatroidAuto}. 
\end{proof}

Let us next investigate matroids which satisfy the weaker property of being modularly complemented. A matroid $M$ is \emph{modularly complemented} if and only if every flat $X$ in $M$ has a modular complement $Y$, i.e.~$Y$ is a modular flat such that $r(X \cup Y) = d+1$ and $X \cap Y = \emptyset$. 

Note that in a connected modularly complemented matroid $M$ of rank $d$,  there exist $d$ modular flats $H_1, \ldots, H_d$ of rank $d-1$ such that $H_1 \cap \ldots \cap H_d = \emptyset$, and such that the intersection of any $d-1$ of them is a element of $E(M)$. 
Let us write 
\[ \{ b_i \}  = H_1 \cap \ldots \cap H_{i-1} \cap H_{i+1} \cap \ldots \cap H_d.\]
Then $b_1, \ldots, b_d$ is a basis of $M$, and every flat generated by a subset of this basis is modular (see \cite{kaku}, p.244).

For matroids of rank at least $4$ there is the following classification of modularly complemented matroids: By \cite{kaku}, a connected simple modularly complemented matroid of rank at least $4$
is either a Dowling matroid $Q_d(G)$ or 
a certain submatroid of the projective geometry $PG(d,q)$.

We will deal with Dowling matroids below. Let us first investigate the case that $M$ is a submatroid of some $PG(d,q)$ for $d \geq 3$ such that the ground set $E$ contains all internal elements of $PG(d,q)$, i.e.~all lines in $\mathbb{F}_q^{d+1}$ contained in one of the coordinate hyperplanes $\{x_i = 0\}$, where $(x_i)_i$ is the dual basis of the canonical basis of $\mathbb{F}_q^{d+1}$.

\begin{proposition}\label{prop:projective} Let $M$ be a submatroid of $PG(d,q)$ for $d \geq 3$ containing all internal elements. Then every automorphism of $B_c(M)$ is induced from a matroid automorphism.
\end{proposition} 

\begin{proof}

Every flat of rank at least $2$ in $PG(d,q)$ corresponds to a  subspace $W$ of $\mathbb{F}_q^{d+1}$ of dimension at least $2$. A straightforward argument shows that $W$ has a basis consisting only of vectors in $\mathbb{F}_q^{d+1}$ which lie in one of the coordinate hyperplanes. Hence every flat of rank at least $2$ in $PG(d,q)$ is also a flat in $M$.

Applying Lemma \ref{lem:coarsemin}, we find that the minimal nested set structure on $B(M)$ coincides with the coarse structure. Let $\varphi$ be an automorphism of $B_c(M)$. By looking at circuits of size $3$, one can check  that $M/i$ is not a non-trivial parallel connection for all $i$. Hence by Corollary \ref{cor:rankcorankone} rays of rank or corank $1$ are mapped by $\varphi$ to rays of rank or corank $1$ (i.e.~rank $d-1$).
 
Let us assume that $\varphi$ sends the corank $1$ ray associated to a coordinate hyperplanes $H_i =\{ x_i = 0\}$ to the rank one ray associated to the element $l$.  To shorten notation, call  this hyperplane $H$. It contains $\frac{q^{d-1}-1}{q-1}$ many one-dimensional subspaces,
which all give rise to internal elements of $PG(d,q)$. Hence there are 
$\frac{q^{d-1}-1}{q-1}$ many rays of rank one in $B_m(M)$ which are neighbors of the ray associated to $H$ in the minimal nested set fan structure. Applying $\varphi$, we find that the ray associated to $l$ has precisely $\frac{q^{d-1}-1}{q-1}$ many neighbor rays of rank $1$ or corank $1$ in the minimal nested set structure. On the other hand, the one-dimensional subspace given by $l$ in $\mathbb{F}_q^{d+1}$ is contained in $\frac{q^{d-1}-1}{q-1}$ hyperplanes. Therefore its ray cannot have any rank $1$ neighbor in the minimal nested set structure. 

Every element of $E(M)$ contained in $H$ gives therefore rise to a ray which is  mapped to  a ray of corank $1$ under $\varphi$.
This in turn implies that all corank $1$ rays are mapped to rank $1$ rays by $\varphi$, since every hyperplane intersects $H$ non-trivially.  Now $M$ contains all corank $1$ flats of $PG(d,q)$.  Since all associated rays map to rank one rays, $M$ must contain all points from $PG(d,q)$. Since this contradicts Proposition \ref{prop:projectivespace}, our assumption must be false, which means that $\varphi$ maps the corank $1$ rays associated to $H_1, \ldots, H_{d+1}$  to corank $1$ rays.

This implies that all rays corresponding to internal points are mapped to rank $1$ rays. It suffices to show that the rays associated to the non-internal points in $E(M)$ are also mapped to rank $1$ rays, since this implies that $\varphi$ is induced by a matroid automorphism. 

If $l$ is any element in $E(M)$, then we find a corank $1$ flat $F$ containing both $l$ and an internal point $p$. If the ray associated to $l$ is mapped to a corank $1$ ray, the ray associated to $F$ is mapped to a rank one ray given by the element $m$. Applying the previous discussion to $\varphi^{-1}$, we find that $m$ is not an internal point. Now the ray associated to $p$ is mapped by $\varphi$ to a rank $1$ ray given by a point $q$. Since $m$ is not internal, the linear span of $q$ and $m$ contains an additional internal vector, hence the rays associated to $m$ and $q$ cannot be neighbors in the minimal nested set structure, contradicting our assumption that the ray for $l$ is sent to a corank one ray. Therefore $\varphi$ is indeed induced by a matroid automorphism. 
\end{proof}

We will now discuss the second class of modularly complemented matroids, namely the Dowling matroids associated to  arbitrary finite groups  $G$. 
Let $d \geq 4$, and put $[d] = \{1, \ldots, d\}$. 
We recall the definition of the Dowling matroid $Q_d(G)$, which is a simple matroid of rank $d$ on 
$E(Q_d(G)) = B \cup \Gamma$
where $B = \{b_1, \ldots, b_d\}$ is the set of coordinate points, and $\Gamma = \{g_{ij}: g \in G, 1 \leq i < j \leq d\}$ is the set of non-coordinate points. For simplicity, let us write $x_1 \vee \ldots \vee x_d$ for the closure of $\{x_1, \ldots, x_d\}$ in $M$. 

Putting $g_{ij} = g_{ji}$ whenever $i >j$, we can list the flats of rank $2$ of $Q_d(G)$ as follows:
The rank $2$ flats are precisely the coordinate flats of the form
\[ b_i \vee b_j = \{b_i, b_j\} \cup \{g_{ij}: g \in G \} \mbox{ for $i \neq j$ }\]
and the non-coordinate flats of the form
\[\{g_{ij}, h_{jk}, (gh)_{ik}\} \mbox{ for $g,  h, gh \in G$  and $i,j,k$ pairwise different in $[d]$}.\]
A subset $F$ of $E(Q_d(G))$ is a flat if and only if it is line-closed, which means that for every two different elements $x,y$  in $F$ the flat $x \vee y$  is also contained in $F$.

Note that a Dowling matroid $Q_d(G)$ is realizable over a field $F$ if and only if $G$ is isomorphic to a subgroup of $F^\ast$, see \cite[Theorem 6.10.10]{oxley}. The matroid automorphisms of $Q_d(G)$ are determined in \cite{bonin}.

Let us now determine the automorphism group the Bergman fan  of a Dowling geometry $Q_d(G)$ equipped with the coarse fan structure. 
In order to determine the minimal nested set structure, we have to study the decomposition of flats into connected components. 
Note that every flat $F$ of $Q_d(G)$ is a product of flats in graphic matroids associated to complete graphs and at most one Dowling matroid on a subset of $B$, see \cite[Proposition 6.10.18]{oxley}. 
In particular the only connected corank $1$ flats  in $Q_d(G)$ are the coordinate flats of the form $b_1 \vee \ldots \vee \widehat{b}_i \vee \ldots b_d$,
i.e.~the closure of $ B \backslash \{b_i\}$, and the non-coordinate flats of  the form $g^{(1)}_{i1} \vee \ldots \vee \widehat{g}^{(i)}_{ii} \vee \ldots \vee g^{(n)}_{in}$ for some $i\in [d]$ and some elements $g^{(1)}, \ldots, g^{(d)}$ of the group $G$.

\begin{prop}\label{prop:dowling}
Let $\varphi:  B_c(Q_d(M)) \rightarrow B_c(Q_d(M))$ be a linear automorphism of the Bergman fan of a Dowling matroid of rank at least $4$ with its coarse structure. Then $\varphi$ is either induced by a matroid automorphism of $Q_d(M)$ or $ \varphi$ is of the form $\varphi = \crem_B \circ \mu$, where $\mu$ is a matroid automorphism $Q_d(M)$ and $\crem_B$ is the Cremona map for the basis $B= \{b_1, \ldots, b_d\}$.
 \end{prop}

\begin{proof}
First of all note that by Theorem \ref{thm:cremiff} the Cremona map $\crem_B$ preserves indeed the support of the Bergman fan.

 The matroid $Q_d(G)$ is obviously not a non-trivial parallel connection, and by \cite[Proposition 6.10.18 (i)]{oxley} we find that for every $x \in E(Q_d(G))$ the simplification of $Q_d(G) / x$ is again a Dowling matroid of rank at least $3$. Hence $Q_d(G) / x$ is not a non-trivial parallel connection, so that we can apply Corollary \ref{cor:rankcorankone} to deduce that $\varphi$ maps any ray of rank $1$ to a ray of rank or corank $1$. 
 
Let $F$ be a connected flat in $Q_d(G)$. By \cite[Proposition 6.10.18 (ii)]{oxley}, the restriction $Q_d(G)|{F}$  is either isomorphic to $Q_r(G)$ for some $r \leq d$ or to the cycle matroid $M(K_s)$ of a complete graph $K_s$. Moreover, if $F$ is any flat of rank $k$ in $Q_r(G)$, by \cite[Section 6.10, exercise 1]{oxley}, the simplification of $Q_r(G) / F$ is isomorphic to $Q_{r-k}(G)$ and hence connected. It is easy to see that for every flat $F$ in $M(K_s)$ the simplification of $M(K_s)/F$ is the cycle matroid associated to a complete graph and hence also connected. We deduce that the matroid $Q_d(G)$ satisfies the criterion in Lemma \ref{lem:coarsemin}, so that the coarse structure on the Bergman fan of $Q_d(G)$ coincides with the minimal nested set structure.
 We consider the following two cases.
 
Case 1: $\varphi$ maps all rays of rank $1$  to rays of rank $1$. Then it is induced by a matroid automorphism of $Q_d(G)$. 

Case 2: There exists a rank $1$ ray mapping to a corank $1$ ray. Let  $F$ be such a flat of rank $1$ such that the associated ray is mapped to the ray associated to the corank $1$ flat $F'$.  We claim that $F$ must be a coordinate flat $\{b_i\}$. Let us first consider the case that $|G| = 1$. Then there exists precisely one non-coordinate corank $1$ flat in $Q_d(G)$, which contains every non-coordinate rank $1$ flat. Such a non-coordinate rank $1$ ray is furthermore adjacent to $d-1 \choose 2$  rays of rank $1$. If $F$ was a non-coordinate rank $1$ flat, we deduce that $F'$ would contain ${d-1 \choose 2}+1$ flats of rank $1$. This is neither satisfied for a coordinate corank $1$ flat nor for the non-coordinate corank $1$ flat. This proves our claim for trivial $G$. Let us now assume that $|G| \geq 2$. If $F$ was a non-coordinate rank $1$ flat, the ray of  any non-coordinate corank $1$ flat containing it would be mapped to a rank $1$ ray by $\phi$.  Every connected non-coordinate corank $1$ flat is of the form  $g^{(1)}_{i1} \vee \ldots \vee \widehat{g}^{(i)}_{ii} \vee \ldots \vee g^{(d)}_{id}$ for some $i$ and some elements $g^{(1)}, \ldots, g^{(d)}$ of the group $G$. This flat contains $ {d \choose 2}$ rank $1$ flats, and hence this number counts the neighbor rays of rank $1$ or corank $1$  in the minimal nested set structure. Since $|G| \geq 2$, a straightforward counting argument shows that for every rank $1$ ray in $Q_d(G)$ the number of neighbor rays of either rank or corank equal to $1$ is strictly bigger than $ {d \choose 2}$. This proves our claim for non-trivial $G$.

Therefore every rank $1$ ray mapping to a corank $1$ ray is indeed associated to some  $\{b_i\}$,
 whereas $\phi$ maps non-coordinate rank $1$ rays to rank $1$ rays. Let $H$ be a coordinate hyperplane containing $b_i$. The ray associated to $H$ is mapped by $\phi$ to a coordinate rank $1$ ray associated to some ${b_k}$. Now $H$ contains ${d-1 \choose 2 }|G|$ non-coordinate rank $1$ flats. Their rays are all mapped to rank $1$ neighbors of the ray given by ${b_k}$. Since  this ray has precisely   ${d-1 \choose 2} |G|$ many rank $1$ neighbors, the rays of all coordinate vectors in $H$ must map to corank $1$ rays.

Hence we deduce that $\varphi$ maps every coordinate rank $1$ ray to a coordinate corank $1$ ray. 

Since every permutation on $[d]$ induces a matroid automorphism of $Q_d(G)$, we deduce that there exists a matroid automorphism $\mu_1$ such that $\varphi \circ \mu_1$ maps each rank $1$ flat $\{b_i \}$ to the corank $1$ flat $b_1 \vee \ldots \vee \widehat{b}_i \vee \ldots b_d$. Hence $\crem_B  \circ \varphi \circ \mu_1$ maps every ray of rank $1$ to itself. Therefore it is induced by a matroid automorphism $\mu_2$, and our claim follows. 
\end{proof}

\author{Kris Shaw,}
\address{Department of Mathematics, 
University of Oslo,  P.O box 1053,
Blindern,
0316, Oslo,
Norway.}
\email{krisshaw@math.uio.no} 

\author{Annette Werner,}
\address{Goethe University Frankfurt, Institut f\"ur Mathematik, Robert-Mayer-Strasse 6-8, 60325 Frankfurt, Germany.}
\email{werner@math.uni-frankfurt.de}
\end{document}